\documentclass{article}

 \PassOptionsToPackage{numbers, compress}{natbib}
\usepackage{amsmath}   % equation environments, \operatorname, \text, \underbrace
\usepackage{amsthm}    % theorem, lemma, proposition, proof environments
\usepackage{amssymb}   % \mathbb, \odot
\usepackage{empheq}    % for \boxed inside equation environments
\newtheorem{theorem}{Theorem}
\newtheorem{lemma}[theorem]{Lemma}

\usepackage{subcaption}
 \usepackage[main, final]{neurips_2026}
\usepackage[utf8]{inputenc} % allow utf-8 input
\usepackage[T1]{fontenc}    % use 8-bit T1 fonts
\usepackage{hyperref}       % hyperlinks
\usepackage{url}            % simple URL typesetting
\usepackage{booktabs}       % professional-quality tables
\usepackage{amsfonts}       % blackboard math symbols
\usepackage{nicefrac}       % compact symbols for 1/2, etc.
\usepackage{microtype}      % microtypography
\usepackage{xcolor}         % colors
\usepackage{algorithm}
\usepackage{algpseudocode}
\usepackage{float}
\hypersetup{hidelinks}

\title{Fast and Stable Gradient Approximation for Quadratic Forms of Hermitian Matrix Functions}

\author{%
  Navjot Singh \\
  Lawrence Berkeley National Laboratory\\
  Berkeley, CA 94720 \\
  \texttt{nsingh2@lbl.gov} \\
  \And
  Kipton Barros \\
  Los Alamos National Laboratory \\
  Los Alamos, NM 87545 \\
  \texttt{kbarros@lanl.gov} \\
  \And
  Xiaoye Sherry Li\\
  Lawrence Berkeley National Laboratory\\
  Berkeley, CA 94720 \\
  \texttt{xsli@lbl.gov} 
}

\begin{document}

\maketitle

\begin{abstract}
Objectives involving quadratic forms \(u^*f(A(\theta))u\) for
Hermitian \(A\) arise widely in scientific computing and machine
learning. Lanczos efficiently approximates these quantities,
but differentiating through its recurrence can introduce
numerical instability and costly reorthogonalization.
Building on existing low-rank Krylov approximations of
Fr\'echet derivatives, we develop a matrix-free gradient
implementation that reuses the forward Lanczos quantities
without differentiating through the recurrence.
We analyze its discrepancy from the exact gradient of the
finite Lanczos approximation, deriving an exact
residual-coupled error identity and a bound involving the
best degree-\((m-1)\) uniform polynomial approximation error
of \(f'\), multiplied by a basis-sensitivity factor.
The discrepancy vanishes when the Krylov subspace is invariant.
Experiments demonstrate accurate gradients and runtime improvements
over differentiated-Lanczos and modified-Arnoldi baselines. It remains numerically stable without
reorthogonalization in all tested settings, including cases
where adjoint-based Lanczos differentiation becomes unstable.
\end{abstract}

\section{Introduction}
Large-scale learning, inverse problems, and scientific computing
often involve objectives built from functions of parameterized
matrices. Examples include log-determinants in Gaussian processes,
matrix exponentials in dynamical systems and network analysis,
and quantum time-evolution operators
\citep{rasmussen2006gaussian,dong2017scalable,immer2021scalable,
saad1992exponential,hochbruck1997krylov,benzi2020,banchi2023learning}.
The matrix \(A(\theta)\) may be too large to form or factorize
densely, while matrix--vector products \(v\mapsto A(\theta)v\)
and their parameter derivatives remain accessible.
Krylov methods provide matrix-free approximations of
matrix-function actions and quadratic forms
\citep{gallopoulos1992parabolic,beckermann2009faber,
ubaru2017fast,chen2023krylov,davies2021krylov}.
For Hermitian \(A(\theta)\), an \(m\)-step Lanczos iteration
approximates
\[
    \phi(\theta;u)=u^*f(A(\theta))u
\]
through a small tridiagonal matrix function.

Computing parameter gradients introduces additional challenges.
Classical Fr\'echet-derivative formulas
\citep{najfeld1995,almohy2009} underpin Krylov methods for
approximating derivative actions and low-rank representations
\citep{kressner2019,kandolf2021,crouzeix2020,kressner2026}.
Alternatively, automatic differentiation or custom adjoint
recurrences can differentiate the finite Lanczos approximation
itself \citep{kramer2024gradients}. These derivatives can be
sensitive to loss of orthogonality even when the forward
quadrature remains accurate, motivating costly
reorthogonalization.

We study a projected gradient approximation that reuses the
forward Lanczos quantities without differentiating through
the recurrence. The method differentiates the small projected
matrix function while holding the Lanczos basis fixed, then
contracts the lifted sensitivity with parameter derivatives
through matrix--vector-product VJPs. The lifted matrix
sensitivity coincides with the low-rank Lanczos
Fr\'echet approximation of \citet{kandolf2021}, within the
broader projection framework of \citet{kressner2019}.
Our focus is its relationship to the exact derivative of
finite Lanczos quadrature and its practical use in
parameter optimization.

Our contributions are:
\begin{itemize}
    \item We formulate a matrix-free parameter-gradient
    implementation using the forward Lanczos quantities
    and matrix--vector-product VJPs, with stored-basis
    and blockwise replay variants.

    \item We derive an exact residual-coupled identity for
    the discrepancy from the derivative of the finite
    Lanczos approximation. We bound this discrepancy by
    the best degree-\((m-1)\) uniform polynomial
    approximation error of \(f'\), multiplied by a
    basis-sensitivity factor, and show that it vanishes
    when the Krylov subspace is invariant.

    \item We demonstrate the accuracy and efficiency of
    the approximation on matrix-logarithm, network-sensitivity,
    and Hamiltonian-learning problems. It remains
    numerically stable without reorthogonalization in
    all tested settings, including cases where
    adjoint-based Lanczos differentiation becomes unstable.
\end{itemize}

\section{Problem setup and model assumptions}
\label{sec:problem}
\subsection{Problem setup}

Let $\theta \in \Theta$ denote the model parameters, and let
$A(\theta) \in \mathbb{R}^{n \times n}$ be a symmetric matrix
depending smoothly on $\theta$. The corresponding Hermitian
formulation uses conjugate transposes and complex polarization;
we restrict to real matrices and vectors for simplicity.
We assume that the spectrum of $A(\theta)$ is contained in
a compact interval $\mathcal I \subset \mathbb{R}$ for all
$\theta$ of interest, and that
\begin{equation*}
f : \mathcal I \to \mathbb{R}
\end{equation*}
is continuously differentiable.

Our primary object of interest is the quadratic form
\begin{equation}
\phi(\theta;u) := u^\top f(A(\theta))u,
\qquad u \in \mathbb{R}^n,
\label{eq:problem_quadratic_form}
\end{equation}
where $u$ is fixed and nonzero.
We focus on \eqref{eq:problem_quadratic_form} because it is
a basic building block for several quantities of interest
involving matrix functions. In particular, trace objectives
can be written as expectations of quadratic forms,
\begin{equation}
\operatorname{tr}(f(A(\theta)))
= \mathbb E_z\!\left[z^\top f(A(\theta)) z\right],
\qquad \mathbb E[zz^\top]=I,
\label{eq:problem_trace_identity}
\end{equation}
for stochastic $z$ with a distribution independent of $\theta$
\citep{hutchinson1989stochastic}.
Bilinear forms can be recovered from quadratic forms by
polarization,
\begin{equation}
u^\top f(A(\theta))v
=
\frac14\Big((u+v)^\top f(A(\theta))(u+v)
-
(u-v)^\top f(A(\theta))(u-v)\Big).
\label{eq:problem_polarization}
\end{equation}
Thus, once we can compute gradients of
\eqref{eq:problem_quadratic_form}, the same machinery also
extends to stochastic trace estimation and more general
objectives involving bilinear forms. Our goal is therefore
to compute the gradient $\nabla_\theta \phi(\theta;u)$
with respect to $\theta$.

\subsection{Matrix-free access model}

We are interested in the large-scale regime in which
$A(\theta)$ is too expensive to form densely or factorize
explicitly. Instead, we assume matrix-free access through
differentiable matrix--vector products:
\begin{equation}
    \mathcal{A}(\theta,u) = A(\theta)u .
    \label{eq:problem_matvec_access}
\end{equation}
This setting covers sparse and structured operators arising
from discretized PDEs and other scientific-computing problems
\citep{saad2003iterative,elman2014finite},
precision and covariance operators in spatial statistics
and Gaussian-process models
\citep{rue2005gaussian,rasmussen2006gaussian},
generalized Gauss--Newton and other curvature matrices in
large-scale optimization
\citep{schraudolph2002fast,martens2010deep},
and more generally any linear operator for which
matrix--vector products are inexpensive even when $n$ is large.

For gradient computation, we additionally assume access to
vector--Jacobian products associated with
$\theta\mapsto\mathcal{A}(\theta,u)$ for fixed $u$.
These may be evaluated analytically or through automatic
differentiation without materializing
$\partial A(\theta)/\partial\theta$.

\section{Background}
\label{sec:background}
\subsection{Lanczos approximation of matrix functions}
\label{subsec:bg_lanczos}

Lanczos iteration is a standard matrix-free tool for approximating
matrix-function-vector products \(f(A)u\) when \(A\) is large and symmetric
\citep{golub2013matrix}. Starting from \(v_1=u/\|u\|\),
\(m\) Lanczos steps construct an orthonormal basis \(V=[v_1,\ldots,v_m]\) for
the Krylov subspace
\[
    \mathcal{K}_m(A,u)=\operatorname{span}\{u,Au,\ldots,A^{m-1}u\},
\]
together with a symmetric tridiagonal matrix \(T\) satisfying
\begin{equation}
    A V = V T + \beta_m v_{m+1} e_m^\top .
    \label{eq:bg_lanczos_relation}
\end{equation}
The resulting Krylov approximation is
\begin{equation}
    f(A)u
    \approx
    \|u\| V f(T)e_1 ,
    \label{eq:bg_lanczos_mvf}
\end{equation}
so the nonlinear function is evaluated only on the small projected matrix \(T\).
The same projection gives the quadratic-form approximation
\begin{equation}
    u^\top f(A)u
    \approx
    \|u\|^2 e_1^\top f(T)e_1 .
    \label{eq:bg_lanczos_quadratic}
\end{equation}
In exact arithmetic, provided no breakdown occurs before step \(m\),
the Lanczos vectors are mutually orthogonal and \(T\) has distinct Ritz
values. In finite precision, however, the three-term recurrence may lose
orthogonality as Ritz values converge, which can lead to ``ghost'' copies
of already converged eigenvalues in \(T\)
\citep{paige1971computation,paige1972computational,paige1976error,parlett1998symmetric}.
This loss of orthogonality does not necessarily make Lanczos-based
quadrature unstable: the computed recurrence can often still be
interpreted as a valid Gauss quadrature rule for a nearby measure, and
the associated quadratic-form estimates are typically robust
\citep{golub2009matrices,meurant2006lanczos}.
Nevertheless, ghost Ritz values can slow or distort convergence.
Full reorthogonalization maintains numerical orthogonality but requires
orthogonalizing each new vector against the previous \(k\) vectors,
costing \(O(nk)\) at step \(k\) and \(O(nm^2)\) overall.
Cheaper alternatives include selective or partial reorthogonalization,
which maintain orthogonality only to the extent needed to suppress
spurious copies of converged Ritz values
\citep{parlett1979lanczos,simons1984analysis}.

When \(A=A(\theta)\), one direct way to compute gradients is to
differentiate the finite Lanczos approximation itself. In automatic
differentiation frameworks, this amounts to backpropagating through
the Krylov iteration. Kr{\"a}mer et al.~\citep{kramer2024gradients}
show that this can be inefficient for Lanczos and Arnoldi iterations,
especially when stable implementations require reorthogonalization.

\subsection{Gradients of matrix functions}
\label{subsec:bg_matrix_function_gradients}

Gradients of objectives involving \(f(A(\theta))\) are naturally
expressed using Fr\'echet derivatives~\citep{higham2008}.
For a matrix function \(f\), the Fr\'echet derivative at \(A\) is
the linear map \(L_f(A,\cdot)\) defined by
\begin{equation}
    f(A+\varepsilon E)
    =
    f(A)+\varepsilon L_f(A,E)+o(\varepsilon),
    \qquad \varepsilon\to 0 .
    \label{eq:bg_frechet_def}
\end{equation}
For symmetric \(A=Q\Lambda Q^\top\), it admits the
divided-difference representation \citep{higham2008}
\begin{equation}
    L_f(A,E)
    =
    Q\bigl(F(\Lambda)\circ (Q^\top E Q)\bigr)Q^\top ,
    \label{eq:bg_hadamard_frechet}
\end{equation}
where \(\circ\) is the elementwise Hadamard product, and
\begin{equation}
    F(\lambda_i,\lambda_j)
    =
    \begin{cases}
        \dfrac{f(\lambda_i)-f(\lambda_j)}{\lambda_i-\lambda_j},
        & \lambda_i\neq\lambda_j,\\[1.2ex]
        f'(\lambda_i), & \lambda_i=\lambda_j.
    \end{cases}
    \label{eq:bg_divided_difference}
\end{equation}
This formula is useful analytically, but it requires spectral
information about the full matrix and dense \(n\times n\)
quantities, making it impractical as a large-scale algorithm.

A common matrix-free alternative is based on the block triangular
identity \citep{higham2008,najfeld1995}
\begin{equation}
    f\!\left(
    \begin{bmatrix}
        A & E\\
        0 & A
    \end{bmatrix}
    \right)
    \begin{bmatrix}
        0\\ u
    \end{bmatrix}
    =
    \begin{bmatrix}
        L_f(A,E)u\\
        f(A)u
    \end{bmatrix}.
    \label{eq:bg_block_action}
\end{equation}
For a parameter direction \(\delta\theta\), setting
\(E=D_\theta A(\theta)[\delta\theta]\) gives the derivative
action \(L_f(A,E)u\) together with the primal action \(f(A)u\).
Even when \(A\) is symmetric, the augmented matrix is generally
nonsymmetric and nonnormal, although its eigenvalues are those
of \(A\), each repeated twice. Thus, the symmetric Lanczos
structure is lost and one typically uses Arnoldi.
Kressner and Oehme exploit the block triangular structure
through a modified Arnoldi method~\citep{kressner2026}.

Another approach is to define a custom reverse-mode rule for
the Lanczos iteration. Kr{\"a}mer et al.~\citep{kramer2024gradients}
formulate Lanczos and Arnoldi iterations as algebraic constraints
and apply the adjoint method to those constraints. For a scalar
loss \(\rho\) depending on the Lanczos outputs, including the
Lanczos vectors and the tridiagonal coefficients, their adjoint
system takes as input the corresponding output sensitivities
and solves a backward recurrence for adjoint variables.
An ambient matrix-gradient representation has the form
\begin{equation}
    \nabla_A \rho
    =
    \sum_{j=1}^m \lambda_j v_j^\top ,
    \label{eq:bg_lanczos_adjoint_matrix_gradient}
\end{equation}
where the adjoint vectors \(\lambda_j\) are obtained from the
backward recurrence. For symmetric perturbations of \(A\),
only the symmetric part of this matrix contributes.
If \(A=A(\theta)\) is accessed through differentiable
matrix--vector products, parameter gradients are computed
by contracting this representation with parameter derivatives
of \(A\), or equivalently through vector--Jacobian products
of the matrix--vector product map.

Adjoint methods avoid generic backpropagation through the
implemented loop, but still differentiate the Lanczos process
through a backward adjoint recurrence and can require
reorthogonalization for numerical stability.
This motivates the question we address next: whether the
projected sensitivity of the small Lanczos objective can be
used without differentiating through the Lanczos recurrence
itself.

\section{Gradient approximation of quadratic forms of matrix functions}
\label{sec:grad_approx}
We derive an efficient approximate gradient of the Lanczos estimate
\begin{equation}
    \widehat{\phi}(\theta;u)
    =
    \|u\|^2 e_1^\top f(T)e_1
    \label{eq:method_lanczos_quad}
\end{equation}
with respect to $\theta$, without backpropagating through the Lanczos recurrence.
Here $T$ is the tridiagonal matrix produced by running $m$ steps of Lanczos on
$A(\theta)$ with initial vector $v_1=u/\|u\|$. For clarity, we treat $\theta$ as a
scalar parameter; the same analysis applies to each parameter direction. Throughout, we write $A=A(\theta)$, and all differentials are
with respect to $\theta$. The derivation below assumes exact arithmetic and a fixed,
nonzero starting vector $u$.

\subsection{Sensitivity with respect to the projected matrix}
\label{subsec:method_T_sensitivity}
\label{sec:sensitivity}

The first ingredient is the derivative of the scalar projected quantity with respect to
its small matrix argument. This is a standard consequence of the Fr\'echet derivative of
a matrix function.

\begin{theorem}[Projected sensitivity]
\label{thm:projected_sensitivity}
Let $T\in\mathbb{R}^{m\times m}$ be symmetric with eigendecomposition
$T=Q\Lambda Q^\top$, let $c=Q^\top e_1$, and let $F=F(\Lambda)$ be the
divided-difference matrix from \eqref{eq:bg_divided_difference}. For
\[
    \widehat{\phi}(T)=\|u\|^2 e_1^\top f(T)e_1,
\]
the differential with respect to $T$ is
\begin{equation}
    d\widehat{\phi}
    =
    \operatorname{tr}(G^\top dT),
    \label{eq:method_dphi_G}
\end{equation}
where
\begin{equation}
    G
    =
    \|u\|^2 Q\bigl((cc^\top)\circ F\bigr)Q^\top .
    \label{eq:method_G}
\end{equation}
\end{theorem}

The matrix $G$ is symmetric and represents sensitivity
to arbitrary symmetric perturbations of the projected
matrix, not only tridiagonal perturbations. Its full
matrix is needed because $V^\top dA\,V$ is generally dense.
The derivation is given in Appendix~\ref{app:sensitivity}.

\subsection{Gradient approximation}
\label{subsec:method_gradient_approx}
\label{sec:gm}

Since $T=V^\top A V$, its full differential is
\begin{equation}
    dT
    =
    \underbrace{V^\top dA\,V}_{\text{direct}}
    +
    \underbrace{dV^\top A V + V^\top A\,dV}_{\text{basis variation}} .
    \label{eq:method_dT_full}
\end{equation}
The direct term captures the explicit dependence of $T$ on $A$, while the remaining
terms capture the dependence of the Lanczos basis $V$ on $\theta$. We approximate
\begin{equation}
    dT \approx V^\top dA\,V,
    \label{eq:method_dT_approx}
\end{equation}
dropping the basis variation terms. Substituting \eqref{eq:method_dT_approx} into \eqref{eq:method_dphi_G} yields
\begin{equation}
    d\widehat{\phi}
    \approx
    \operatorname{tr}\!\left(G V^\top dA\,V\right).
    \label{eq:method_grad_approx}
\end{equation}

The lifted sensitivity $VGV^\top$ coincides with the low-rank
Lanczos approximation of $L_f(A,uu^\top)$ developed by
\citet{kandolf2021}. Here we analyze the discrepancy between its parameter
contraction and the exact derivative of finite Lanczos
quadrature, and evaluate its use in matrix-free optimization.
The following theorem compares our approximation with the exact
derivative of the finite $m$-step Lanczos estimate. To distinguish
these two quantities, define
\begin{equation}
    g_{\mathrm{exact}}
    :=
    \frac{d\widehat{\phi}(\theta;u)}{d\theta},
    \qquad
    g_{\mathrm{approx}}
    :=
    \operatorname{tr}\!\left(VGV^\top E\right),
    \qquad
    E:=\frac{\partial A(\theta)}{\partial\theta}.
    \label{eq:method_exact_approx_gradients}
\end{equation}
Let $\Pi_k$ denote the set of polynomials of degree at most $k$. Let $\mathcal I=[a,b]\subset\mathbb{R}$ be a compact interval containing all the eigenvalues of $A$, $\sigma(A)$,  for all parameter values under consideration.
Since $T_m=V_m^\top AV_m$, its Ritz values also lie in $\mathcal I$.
Define the best uniform polynomial approximation error
\begin{equation}
    \mathcal{E}_k(h;\mathcal I)
    :=
    \min_{q\in\Pi_k}
    \max_{x\in \mathcal I}|h(x)-q(x)|.
    \label{eq:method_best_poly_error}
\end{equation}

\begin{theorem}[Error relative to the exact Lanczos gradient]
\label{thm:gradient_convergence}
Suppose that $m$ steps of exact-arithmetic Lanczos, with
fixed starting vector $u/\|u\|$, produce $V$ and $T$,
with no breakdown before step $m$. Define the final residual
\begin{equation}
    r_m := Av_m-VTe_m .
    \label{eq:method_final_residual}
\end{equation}
When $r_m\neq0$, we have $r_m=\beta_m v_{m+1}$;
a final breakdown with $r_m=0$ is allowed. Assume that $f\in C^1(\mathcal I)$ and that $\sigma(A)\subset \mathcal I$. Then
\begin{equation}
    g_{\mathrm{exact}}-g_{\mathrm{approx}}
    =
    2e_m^\top G
    \left(\frac{\partial V}{\partial\theta}\right)^\top r_m .
    \label{eq:method_exact_error}
\end{equation}
Moreover,
\begin{equation}
\begin{aligned}
    \left|
    g_{\mathrm{exact}}-g_{\mathrm{approx}}
    \right|
    \leq
    2\|u\|^2
    \left\|
    \left(\frac{\partial V}{\partial\theta}\right)^\top r_m
    \right\|_2
    \mathcal{E}_{m-1}(f';\mathcal I).
\end{aligned}
    \label{eq:method_gradient_convergence_bound}
\end{equation}
Consequently, the discrepancy tends to zero whenever
$\mathcal{E}_{m-1}(f';\mathcal I)\to0$ and the
residual-coupled sensitivity
$\|(\partial_\theta V)^\top r_m\|_2$ remains uniformly
bounded along the sequence considered. If the Krylov subspace is $A$-invariant, then $r_m=0$ and
$g_{\mathrm{exact}}=g_{\mathrm{approx}}$ exactly.
\end{theorem}

Polynomial approximation of $f'$ also controls errors in
Fr\'echet-action approximations \citep{kressner2026}.
Our bound instead quantifies the discrepancy from the exact
finite-Lanczos derivative through the residual-coupled
basis sensitivity. 
For a fixed maximum Lanczos depth, this sensitivity is
bounded on compact parameter sets uniformly separated
from intermediate breakdown; see
Appendix~\ref{app:gradient_convergence} for the proof
and the complete theorem derivation.
The result concerns the derivative of the finite Lanczos
estimate, rather than directly bounding error relative
to $\partial_\theta\phi$. The theorem is an exact-arithmetic result; numerical
behavior without reorthogonalization is assessed
experimentally.

\subsection{Practical implementation}
\label{subsec:method_implementation}

The approximation in \eqref{eq:method_grad_approx} can be evaluated without forming
the dense matrix $VGV^\top$. We describe two implementations. The first stores the
Lanczos basis and is simplest. The second replays the Lanczos recurrence and avoids
storing the full basis.

\paragraph{Stored-basis implementation.}
Run $m$ steps of Lanczos with $A(\theta)$ and initial vector $v_1=u/\|u\|$ to obtain
\[
    V=[v_1,\ldots,v_m]\in\mathbb{R}^{n\times m},
    \qquad
    T\in\mathbb{R}^{m\times m}.
\]
Compute $T=Q\Lambda Q^\top$, set $c=Q^\top e_1$, form the divided-difference matrix
$F$, and construct $G$ as in~\eqref{eq:method_G}. Then
\[
    \frac{d\widehat{\phi}}{d\theta}
    \approx
    \operatorname{tr}\!\left(GV^\top
    \frac{\partial A(\theta)}{\partial\theta}V\right).
\]
Equivalently, if $W=VG=[w_1,\ldots,w_m]$, then
\begin{equation}
    \frac{d\widehat{\phi}}{d\theta}
    \approx
    \sum_{j=1}^{m}
    w_j^\top
    \frac{\partial A(\theta)}{\partial\theta}
    v_j .
    \label{eq:method_grad_vjp}
\end{equation}
Each term is evaluated as a parameter VJP through
$\theta\mapsto\mathcal A(\theta,v_j)$, with $v_j$ held
fixed and $w_j$ supplied as a fixed cotangent.
Thus, no derivatives are propagated through $V$, $G$,
or the Lanczos recurrence.

\paragraph{Blockwise replay implementation.}
For large problems, storing $V$ can dominate the memory and affect practical performance significantly. Instead, we first run unreorthogonalized Lanczos and retain
the starting vector and the tridiagonal coefficients defining $T$. After forming
$G$ from $T$, we replay the same Lanczos recurrence. For each target block $B$ of at most $b$ column indices,
we replay the full unreorthogonalized three-term recurrence
using the stored coefficients. As each vector $v_i$ is
regenerated, we accumulate
$W_B=\sum_{i=1}^m v_iG_{i,B}$ and retain the vectors
whose indices belong to $B$ to form $V_B$. Once $W_B$ and the matching block $V_B$ are available, we evaluate the VJPs
\[
    \sum_{j\in B}
    [D_\theta\mathcal{A}(\theta,v_j)]^\top w_j .
\]
The block is then discarded and the next block is processed. This gives a direct
tradeoff: larger $b$ uses fewer replay passes, while smaller $b$ reduces memory.
When the derivative of the matrix-vector product has structure, as in graph
Laplacians or local discretized operators, these VJP contractions can be batched or
assembled directly from structured edge/element contributions, avoiding dense
projected derivative matrices.

We summarize the stored-basis version in Algorithm~\ref{alg:method_gradient}.
The blockwise replay variant replaces Lines 4--5 by blockwise replay contractions and also changes Line 1 to not retain $V$.

\begin{algorithm}[H]
\caption{Approximate gradient of $\phi(\theta;u)=u^\top f(A(\theta))u$}
\label{alg:method_gradient}
\begin{algorithmic}[1]
\Require Parameter $\theta$, vector $u$, Lanczos steps $m$, function $f$ and derivative $f'$
\Ensure Approximate gradient $\nabla_{\theta} \phi(\theta ; u)$
\State Run $m$-step Lanczos on $A(\theta)$ with $v_1=u/\|u\|$ to obtain $V$ and $T$
\State Compute $T=Q\Lambda Q^\top$ and set $c\leftarrow Q^\top e_1$
\State Form the divided-difference matrix $F$ using $f$ and $f'$ as in~\eqref{eq:bg_divided_difference}, and set
\[
    G \leftarrow \|u\|^2 Q\bigl((cc^\top)\circ F\bigr)Q^\top .
\]
\State Compute $W\leftarrow VG$
\State \Return $\displaystyle
   \sum_{j=1}^{m}
[D_\theta\mathcal{A}(\theta,v_j)]^\top w_j$
\end{algorithmic}
\end{algorithm}

\paragraph{Complexity and memory.}
With stored basis, the dominant large-scale costs are $m$ matrix-vector products,
$m$ VJPs, and the dense product $VG$, which costs $O(nm^2)$ but uses optimized dense
linear algebra. The other small operations cost $O(m^3)$ for forming $G$ after
the eigendecomposition of $T$. Storing $V$, $W$, and the projected quantities requires
$O(nm+m^2)$ working storage, excluding operator,
parameter-gradient, and VJP workspace.

With replay, a block size $b$ requires $O(nb+m^2)$
working storage, excluding operator, parameter-gradient,
and VJP workspace. Approximately $\lceil m/b\rceil$
full replay passes add
$O(m\lceil m/b\rceil)$ matrix--vector products.
Accumulating all blocks of $VG$ still costs $O(nm^2)$,
and the method evaluates $m$ parameter VJPs.
Thus, replay reduces vector storage at the cost of
additional matrix--vector products.

\section{Experiments}
\label{sec:experiments}
% Integrated revision with complete graph timings and polarized graph sensitivities.
We evaluate gradient accuracy in three settings:
matrix-logarithm quadratic forms, compared with adjoint-based Lanczos
differentiation~\citep{kramer2024gradients}; graph sensitivities of
matrix-exponential quadratic forms, compared with the modified-Arnoldi
method of \citet{kressner2026}; and Hamiltonian learning, compared with
a dense Fr\'echet-derivative baseline. The main-text experiments were run on a MacBook Pro with
an Apple M4 Pro chip (12-core CPU: 8 performance cores and
4 efficiency cores), 24 GB unified memory, running
macOS 26.2 on arm64~\footnote{Code is publicly available on \href{https://github.com/navjo2323/Gradient-Lanczos-NeurIPS}{Github}}. Appendix~\ref{app:lanczos-regimes}
presents an additional matrix-free GPU experiment on an affine
fractional-covariance model, comparing gradient accuracy and runtime as
the Lanczos depth increases and discussing implementation storage costs.

\paragraph{Implementation details.}
For the matrix-logarithm experiment, we implement our projected gradient
within the same \texttt{matfree}-based codebase used by the Lanczos-adjoint
baseline of \citet{kramer2024gradients}. Thus, our method and the adjoint baseline share
the same Lanczos forward pass, matrix-vector-product interface, objectives, and reference
routines; only the gradient computation differs. For the modified block-Arnoldi comparisons, we
use the public \texttt{fAb}-Frechet repository accompanying \citet{kressner2026}, reusing
its graph-loading code, network-sensitivity objectives, and reference computations.
Sparse matrix-vector products are evaluated with the NumPy/SciPy sparse routines used
by that implementation.

\subsection{Matrix logarithms and their derivatives}
\label{subsec:experiments_logdet}

Log-determinants of symmetric positive definite matrices arise in Gaussian
process marginal likelihoods and other models involving covariance,
precision, Hessian, or Gauss--Newton matrices
\citep{rasmussen2006gaussian,gardner2018gpytorch,ubaru2017fast,dong2017scalable,immer2021scalable}.
For an SPD matrix $K(\theta)$,
\[
    \log\det K(\theta)=\operatorname{tr}(\log K(\theta)),
\]
and stochastic Lanczos quadrature estimates this trace using random-probe
quadratic forms. Here we use a single fixed Rademacher probe $u$ and
compare the approximate quadratic form and its gradient with their
dense references.

We construct $K(\theta)$ from a synthetic Gaussian-process regression problem with
Gaussian inputs $x_i\in\mathbb{R}^2$ generated with seed $0$ and an
RBF kernel with diagonal noise,
\[
K_{ij}(\theta)
=
\sigma_f^2
\exp\!\left(
-\frac{1}{2}\left\|\frac{x_i-x_j}{\ell}\right\|_2^2
\right)
+
(\sigma_n^2+10^{-6})\delta_{ij}.
\]
The parameters are initialized as
$\ell=0.9$, $\sigma_f=1.1$, and $\sigma_n=0.15$. We use
$n=15{,}000$, and compare against a dense eigendecomposition reference for probe
value and gradient. Gradients use the model's log lengthscale,
log outputscale, log noise, and mean coordinates; the mean component
is zero for this quadratic-form objective.
% VERIFY LOG_METADATA: historical precision and matfree/JAX versions were
% not established by the evidence report. Preserve the run's own metadata.

Figure~\ref{fig:trlogdet-adjoint} compares our projected gradient with
the Lanczos-adjoint implementation in \texttt{matfree}~\citep{kramer2024gradients}.
Both methods use the same forward Lanczos approximation, evaluated
with and without full reorthogonalization. The figure reports errors
against the dense reference. With full reorthogonalization, the
projected and adjoint gradients also agree closely at moderate depths:
their Euclidean difference, normalized by the dense-reference gradient
norm, is approximately $3.4\times10^{-5}$ at $m=60$ and
$2.4\times10^{-8}$ at $m=80$. The forward relative errors are below
$10^{-6}$ and approximately $10^{-11}$, respectively. At larger
depths, the adjoint gradient becomes unstable.
Without reorthogonalization, the forward approximation
converges more slowly, while the projected gradient
remains accurate even where the adjoint gradient deteriorates.

Appendix~\ref{app:Gpy_logdet} compares end-to-end GP training
with GPyTorch; Appendix~\ref{app:lanczos-regimes} compares
the affine specialization with AD and Chebyshev baselines across
Lanczos depths.

\begin{figure}[t]
    \centering
    \includegraphics[width=0.9\linewidth]{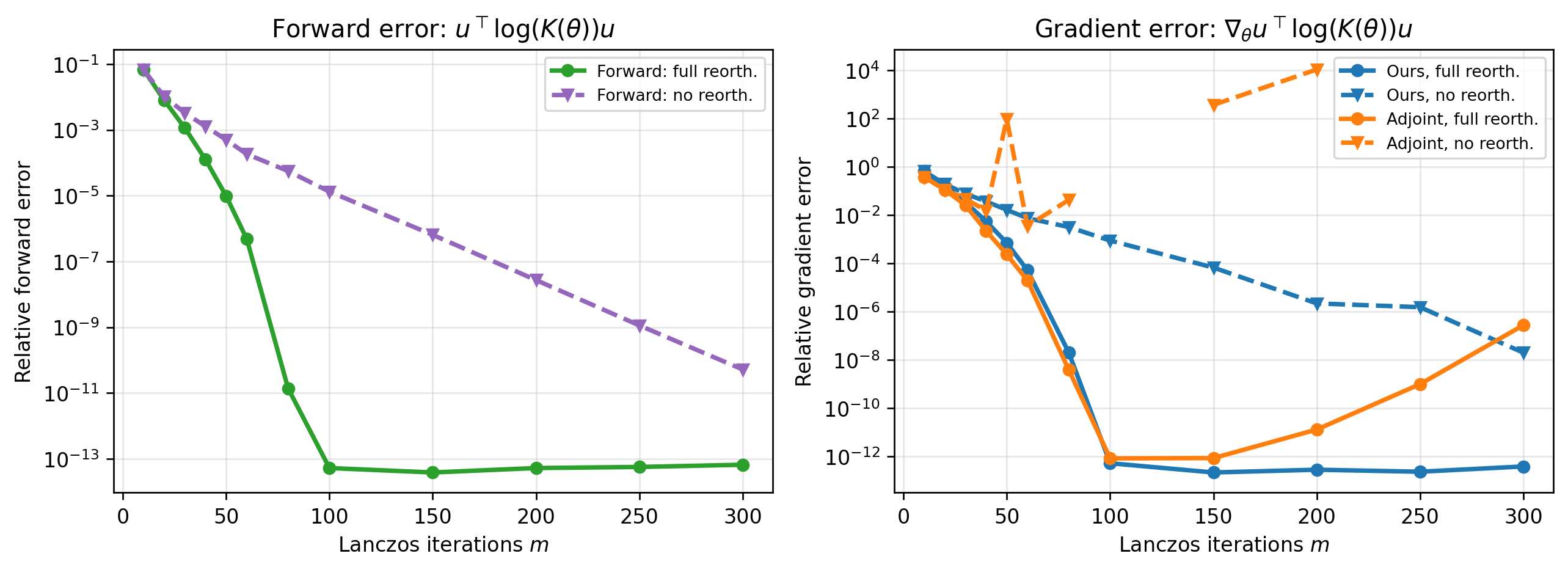}
    \caption{
    Projected and adjoint-based Lanczos gradients~\citep{kramer2024gradients}
    for $u^\top\log(K(\theta))u$ with $n=15{,}000$ and one fixed
    Rademacher probe. Left: relative error of the quadratic-form
    approximation. Right: relative gradient error with respect to the
    kernel hyperparameters. Both panels use dense references.
    }
    \label{fig:trlogdet-adjoint}
\end{figure}

\subsection{Network sensitivity benchmarks}

We next evaluate our Lanczos gradient approximation on graph-sensitivity tasks based on
matrix exponentials. In applications, these sensitivities can
be used to rank influential edges or interactions, identify which parts of a network most
affect information flow, and assess the robustness of network statistics to structural
changes \citep{benzi2020,delacruz2022,pozza2018,schweitzer2023}. Following the
network-analysis setup of \citet{kressner2026}, we focus on sensitivities derived from
the matrix exponential on sparse graph adjacency matrices.

We consider two objectives. The first is the \emph{total-network-communicability} (TN)
sensitivity
\[
S^{(\mathrm{TN})}_{ij}(A)
=
e_i^\top L_{\exp}(A^\top,\mathbf{1}\mathbf{1}^\top)e_j,
\]
which measures how a global communicability score changes under perturbations of the
graph. The second is the \emph{subgraph-centrality} (SC) sensitivity
\[
S^{(\mathrm{SC})}_{ij}(A,\ell)
=
e_i^\top L_{\exp}(A^\top,e_\ell e_\ell^\top)e_j,
\]
which measures the sensitivity of the local subgraph-centrality score associated with
node $\ell$. These are the entrywise derivatives of
$\mathbf1^\top e^A\mathbf1$ and $e_\ell^\top e^Ae_\ell$, respectively.
For a shared off-diagonal edge weight in an undirected graph, the
derivative is twice the corresponding entrywise sensitivity.
We report entrywise sensitivities throughout.

The implementation uses polarization with $z_\pm=e_i\pm e_j$.
Two independent Lanczos recurrences, evaluated in a batch, approximate
the quadratic forms for these starts. Their values and projected
sensitivities are combined with weights $+1/4$ and $-1/4$.
For each sign $\sigma$, the contraction with $E=ww^\top$ is
\begin{equation}
 \operatorname{tr}\!\left(G_\sigma V_\sigma^\top ww^\top V_\sigma\right)
 =(V_\sigma^\top w)^\top G_\sigma(V_\sigma^\top w),
 \label{eq:exp_rank_one_contraction}
\end{equation}
where $w=\mathbf1$ for TN and $w=e_\ell$ for SC.
Once $V_\sigma$ and $G_\sigma$ are available, these contractions cost
$O(nm+m^2)$ for TN and $O(m^2)$ for SC. The latter uses only row
$\ell$ of each basis.

We compare our gradient approximation against the modified-Arnoldi
method of \citet{kressner2026} (referred to as KO), using a sparse block-exponential computation as the
reference for error measurement. We run these experiments on sparse undirected graphs from the SNAP collection: \texttt{ca-HepTh}, an
arXiv high-energy-physics theory coauthorship network, and \texttt{com-Amazon}, an
Amazon product co-purchasing network with ground-truth communities. Additional dataset
and implementation details are given in Appendix~\ref{app:network-sensitivity}.

Table~\ref{tab:network-sensitivity-summary} compares both graphs and
objectives at $m=16,32$. With complete value-and-gradient evaluations
timed for both methods, the batched projected implementation achieves
mean paired speedups of $1.37$--$2.69\times$ over KO.
Timing details and absolute runtimes are given in
Appendix~\ref{app:network-sensitivity}.
At $m=32$, both methods achieve high accuracy, with several results
close to machine precision. At $m=16$, our method is comparable or
more accurate except on \texttt{ca-HepTh}/SC.
Full curves appear in Appendix~\ref{app:network-sensitivity}.

\begin{table}[t]
\centering
\small
\caption{
Graph-sensitivity benchmarks at Krylov depths $m=16,32$.
TN and SC denote total network communicability and subgraph centrality;
KO denotes the modified-Arnoldi method of \citet{kressner2026}.
Left: mean paired speedups from three complete value-and-gradient
timing runs. Right: recorded relative errors in the selected
entrywise sensitivities.
}
\label{tab:network-sensitivity-summary}

\begin{minipage}[t]{0.36\linewidth}
\centering
\caption*{\small Speedup}
\begin{tabular}{llcc}
\toprule
Graph & Obj. & $16$ & $32$ \\
\midrule
HepTh & TN & $1.37$ & $1.44$ \\
Amazon & TN & $2.28$ & $2.68$ \\
HepTh & SC & $1.40$ & $1.47$ \\
Amazon & SC & $2.19$ & $2.69$ \\
\bottomrule
\end{tabular}
\end{minipage}
\hfill
\begin{minipage}[t]{0.60\linewidth}
\centering
\caption*{\small Relative sensitivity error}
\begin{tabular}{llcccc}
\toprule
& & \multicolumn{2}{c}{$m=16$} & \multicolumn{2}{c}{$m=32$} \\
\cmidrule(lr){3-4}\cmidrule(lr){5-6}
Graph & Obj. & Ours & KO & Ours & KO \\
\midrule
HepTh & TN & $10^{-4}$ & $10^{-4}$ & $10^{-14}$ & $10^{-11}$ \\
Amazon & TN & $10^{-7}$ & $10^{-5}$ & $10^{-14}$ & $10^{-12}$ \\
HepTh & SC & $10^{-1}$ & $10^{-2}$ & $10^{-10}$ & $10^{-12}$ \\
Amazon & SC & $10^{-8}$ & $10^{-8}$ & $10^{-14}$ & $10^{-15}$ \\
\bottomrule
\end{tabular}
\end{minipage}
\end{table}

\subsection{Quantum Hamiltonian learning}
\label{subsec:experiments_quantum}

Hamiltonian learning seeks to recover a quantum system's Hamiltonian from
observations of its dynamics, with applications in quantum-device calibration
and simulator validation
\citep{banchi2023learning,huang2023robust,haah2024practical}.
We consider an $L$-site Ising-type Hamiltonian
\[
    H_\theta = \sum_{j=1}^{3L-1} \theta_j P_j,
    \qquad
    P_j \in \{X_i, Z_i\}_{i=0}^{L-1}
    \cup \{Z_iZ_{i+1}\}_{i=0}^{L-2},
\]
where $X_i$ and $Z_i$ are single-site Pauli operators.
For $L=8$, the Hamiltonian has dimension $256$ and $23$ unknown
parameters. We generate $30$ noiseless state pairs satisfying
$c_j=\exp(-\mathrm{i}t_jH_{\theta_\star})x_j$ and minimize
\[
    \mathcal{L}(\theta)
    =
    \sum_{j=1}^{30}
    \left\|c_j-\exp(-\mathrm{i}t_jH_\theta)x_j\right\|_2^2.
\]
Starting from the same perturbation of $\theta_\star$, both methods
use $400$ Adam updates and evaluate the forward loss by dense matrix
exponentiation. The baseline computes the dense Fr\'echet derivative;
our method approximates the gradient using polarized cosine and sine
quadratic-form sensitivities. Data generation and implementation
details are given in Appendix~\ref{app:hamiltonian-learning}.

\begin{table}[t]
\centering
\small
\caption{
Hamiltonian learning for $L=8$ after $400$ optimization steps.
Gradient relative error is measured at initialization against the
dense matrix-exponential Fr\'echet derivative. Final parameter error
is $\|\theta-\theta_\star\|_2$. Speedups compare accumulated
loss-and-gradient evaluation times, including dense forward
propagation and excluding optimizer updates.
}
\label{tab:hamiltonian-learning}
\begin{tabular}{cccc}
\toprule
Method & Grad. rel. err. & Final param. error & Speedup \\
\midrule
Dense & -- & $2.00{\times}10^{-10}$ & $1.0{\times}$ \\
$m=30$ & $7.49{\times}10^{-2}$ & $5.43{\times}10^{-3}$ & $14.8{\times}$ \\
$m=40$ & $6.31{\times}10^{-4}$ & $1.61{\times}10^{-4}$ & $12.4{\times}$ \\
$m=50$ & $5.62{\times}10^{-7}$ & $1.05{\times}10^{-7}$ & $10.4{\times}$ \\
\bottomrule
\end{tabular}
\end{table}

Table~\ref{tab:hamiltonian-learning} shows that increasing the
Krylov depth improves both gradient accuracy and parameter recovery.
At $m=50$, the method achieves a final parameter error of
$1.05{\times}10^{-7}$, with a $10.4\times$ speedup in the accumulated
loss-and-gradient evaluation time relative to the dense baseline. Smaller depths provide greater speedups at the
cost of recovery accuracy. Loss and parameter-error trajectories versus optimization
step appear in Appendix~\ref{app:hamiltonian-learning}.

\section{Limitations and future work}
This work focuses on Hermitian matrices, for which Lanczos provides a
three-term recurrence and a tridiagonal projected matrix. Extending
the analysis to general non-Hermitian matrices is a direction for future
work. Our analysis concerns scalar quadratic forms; for derivatives of
matrix-function-vector products, derivative-Krylov or adjoint-based
methods may be preferable. The stored-basis implementation can become
memory-intensive for large matrices or Krylov depths. Reducing storage
and recomputation costs is another direction for future work. Finally,
our error analysis assumes exact arithmetic; analysis of the
finite-precision implementation remains open.

\section*{Acknowledgement}
This material is based upon work supported by the U.S. Department of Energy, Office of Science, Office of
Advanced Scientific Computing Research and Office of Basic Energy Sciences, Scientific Discovery through
Advanced Computing (SciDAC) program.

\bibliographystyle{plainnat}
\bibliography{sections/references}

\appendix

\section{Proofs for Section~\ref{sec:grad_approx}}
\label{app:section4_details}
We prove the projected-sensitivity formula and the error bound in
Section~\ref{sec:grad_approx}. All statements in this section concern exact
arithmetic and a fixed Lanczos depth.

Let $A(\theta)\in\mathbb R^{n\times n}$ be symmetric and continuously
differentiable on an open parameter interval, and write
$E=\partial_\theta A$. Let $u\neq0$ be fixed, and let $f$ be real-valued and continuously
differentiable on a neighborhood of $\mathcal I=[a,b]$, which contains the
spectra under consideration. Suppose Lanczos, initialized with
$v_1=u/\|u\|_2$, has no breakdown before step $m$. Then
$V=[v_1,\ldots,v_m]$ and $T=V^\top AV$ are differentiable locally, and
\begin{equation}
 AV=VT+r_me_m^\top,
 \qquad r_m=Av_m-VTe_m,
 \qquad V^\top r_m=0.
 \label{eq:app_lanczos_relation}
\end{equation}
When $r_m\neq0$, write $r_m=\beta_m v_{m+1}$ with
$\beta_m=\|r_m\|_2$. A final breakdown $r_m=0$ is allowed: neither the
construction of $V$ nor the proof requires differentiating $v_{m+1}$.

Define
\begin{equation}
 \widehat\phi(\theta;u)=\|u\|_2^2e_1^\top f(T)e_1,
 \qquad
 g_{\mathrm{exact}}=\partial_\theta\widehat\phi,
 \qquad
 g_{\mathrm{approx}}=\operatorname{tr}(VGV^\top E),
 \label{eq:app_exact_approx}
\end{equation}
where $G$ is derived below. In particular, $g_{\mathrm{exact}}$ denotes the
derivative of the finite Lanczos estimate, not necessarily the derivative
of $u^\top f(A)u$. For ease of writing, we also use
\begin{equation}
 \xi_m=(\partial_\theta V)^\top r_m.
 \label{eq:app_xi_definition}
\end{equation}

\begin{lemma}[Vanishing commutator]
\label{lem:app_commutator}
Let $G$ be the projected sensitivity defined in \eqref{eq:method_G}.
If $S^\top=-S$ and $Se_1=0$, then
\[
    \operatorname{tr}(G^\top[T,S])=0.
\]
\end{lemma}

\begin{proof}[Original proof via divided differences]
Write $T=Q\Lambda Q^\top$, set $c=Q^\top e_1$, and define
$\widetilde S=Q^\top S Q$. Orthogonality of $Q$ gives
\[
    \widetilde S^\top=-\widetilde S,
    \qquad
    \widetilde S c=Q^\top Se_1=0,
\]
and
\[
    Q^\top[T,S]Q
    =\Lambda\widetilde S-\widetilde S\Lambda.
\]
Using \eqref{eq:method_G}, symmetry of $(cc^\top)\circ F(\Lambda)$,
and cyclic invariance of the trace, we obtain
\[
    \operatorname{tr}(G^\top[T,S])
    =\|u\|_2^2 C,
\]
where
\[
    C=
    \operatorname{tr}\!\left(
        \bigl((cc^\top)\circ F(\Lambda)\bigr)
        \bigl(\Lambda\widetilde S-\widetilde S\Lambda\bigr)
    \right).
\]
Expanding entrywise yields
\[
    C=\sum_{i,j}
      c_i c_j F(\lambda_i,\lambda_j)
      (\lambda_j-\lambda_i)\widetilde S_{ji}.
\]
The divided-difference identity
\[
    F(\lambda_i,\lambda_j)(\lambda_j-\lambda_i)
    =f(\lambda_j)-f(\lambda_i)
\]
holds for distinct eigenvalues; when the eigenvalues coincide,
both sides vanish. Therefore,
\begin{align*}
    C
    &=\sum_{i,j}c_i c_j
      \bigl(f(\lambda_j)-f(\lambda_i)\bigr)
      \widetilde S_{ji}\\
    &=(f(\Lambda)c)^\top\widetilde S c
      -c^\top\widetilde S f(\Lambda)c.
\end{align*}
Since $\widetilde S$ is skew-symmetric,
\[
    c^\top\widetilde S f(\Lambda)c
    =-(f(\Lambda)c)^\top\widetilde S c.
\]
Consequently,
\[
    C=2(f(\Lambda)c)^\top\widetilde S c=0,
\]
where the last equality follows from $\widetilde S c=0$.
Thus $\operatorname{tr}(G^\top[T,S])=0$.
\end{proof}

\begin{proof}[AI-aided proof via orthogonal invariance]
Define $U(t)=\exp(tS)$ for $t\in\mathbb{R}$. Since $S$ is
skew-symmetric,
\[
    U(t)^\top=\exp(-tS),
    \qquad
    U(t)^\top U(t)=I.
\]
Thus $U(t)$ is orthogonal. Moreover, $Se_1=0$ implies
$S^k e_1=0$ for every $k\geq1$, so
\[
    U(t)e_1
    =
    \left(I+\sum_{k=1}^{\infty}\frac{t^kS^k}{k!}\right)e_1
    =e_1.
\]
The spectral definition of a symmetric matrix function gives
\[
    f(U(t)^\top T U(t))
    =U(t)^\top f(T)U(t).
\]
Consequently,
\[
    e_1^\top f(U(t)^\top T U(t))e_1
    =e_1^\top f(T)e_1,
\]
so this quadratic form is constant along the curve
$U(t)^\top T U(t)$. Since $U(0)=I$ and $U'(0)=S$,
\[
    \left.\frac{d}{dt}
    \bigl(U(t)^\top T U(t)\bigr)\right|_{t=0}
    =S^\top T+TS
    =TS-ST
    =[T,S].
\]
Applying the projected-sensitivity formula yields
\[
    0
    =
    \left.\frac{d}{dt}
    \left(\|u\|_2^2 e_1^\top
    f(U(t)^\top T U(t))e_1\right)\right|_{t=0}
    =
    \operatorname{tr}(G^\top[T,S]).
\]
\end{proof}

\subsection{Polynomial control of the gradient error}
\label{app:gradient_convergence}
For the space $\Pi_k$ of polynomials of degree at most $k$, define
\[
 \mathcal E_k(h;\mathcal I)=\min_{q\in\Pi_k}
       \|h-q\|_{\infty,\mathcal I}.
\]

\begin{lemma}[Boundary polynomial annihilation]
\label{lem:app_boundary_annihilation}
Let $T\in\mathbb R^{m\times m}$ be symmetric tridiagonal and let
$z\in\mathbb R^m$ satisfy $e_1^\top z=0$. Then
\begin{equation}
 e_m^\top L_p(T,e_1e_1^\top)z=0
 \qquad\text{for every }p\in\Pi_m.
 \label{eq:app_boundary_annihilation}
\end{equation}
\end{lemma}
\begin{proof}
The derivative of a constant is zero. For $p(x)=x^k$, $1\leq k\leq m$,
\[
 e_m^\top L_p(T,e_1e_1^\top)z
 =\sum_{\ell=0}^{k-1}
   (e_m^\top T^\ell e_1)(e_1^\top T^{k-1-\ell}z).
\]
Tridiagonality implies $e_m^\top T^\ell e_1=0$ for $\ell<m-1$.
The only remaining possibility is $k=m$, $\ell=m-1$, whose second
factor is $e_1^\top z=0$. Linearity proves the result.
\end{proof}

Choose any $q\in\Pi_{m-1}$ and an antiderivative $p\in\Pi_m$ with
$p'=q$. Since $e_1^\top\xi_m=0$, the lemma gives
\[
 e_m^\top L_f(T,e_1e_1^\top)\xi_m
 =e_m^\top L_h(T,e_1e_1^\top)\xi_m,\qquad h=f-p.
\]
The Ritz values lie in $\mathcal I$. With $c=Q^\top e_1$ and
$\|c\|_2=1$, the divided-difference representation implies
\begin{align*}
 \|L_h(T,e_1e_1^\top)\|_F^2
 &=\sum_{i,j}|h^{[1]}(\lambda_i,\lambda_j)|^2|c_i|^2|c_j|^2\\
 &\leq\|h'\|_{\infty,\mathcal I}^2
       \left(\sum_i|c_i|^2\right)^2
 =\|f'-q\|_{\infty,\mathcal I}^2.
\end{align*}
By the fundamental theorem of calculus,
\[
 h^{[1]}(x,y)=\int_0^1h'((1-t)y+tx)\,dt.
\]
Since $h'$ is continuous, this expression extends continuously to
$x=y$, where it equals $h'(x)$. Hence
\[
 |e_m^\top L_f(T,e_1e_1^\top)\xi_m|
 \leq\|\xi_m\|_2\|f'-q\|_{\infty,\mathcal I}.
\]
Minimizing over $q$ and using \eqref{eq:method_G} and
\eqref{eq:app_exact_gradient_decomposition} proves
\begin{equation}
 |g_{\mathrm{exact}}-g_{\mathrm{approx}}|
 \leq 2\|u\|_2^2\|\xi_m\|_2\,
             \mathcal E_{m-1}(f';\mathcal I).
 \label{eq:app_gradient_final_bound}
\end{equation}

For fixed $f\in C^1(\mathcal I)$ and compact $\mathcal I$, the
Weierstrass approximation theorem gives
$\mathcal E_{m-1}(f';\mathcal I)\to0$ as $m\to\infty$.
Convergence of the full bound also depends on $\|\xi_m\|_2$.
The following lemma bounds this factor over a fixed finite range
of Lanczos depths.

\begin{lemma}[Boundedness of the finite-step Lanczos sensitivity]
\label{lem:finite_lanczos_sensitivity}
Let $\Theta_0$ be compact and let $A$ be continuously differentiable
on an open neighborhood of $\Theta_0$. Fix $M\leq n$, and suppose
\[
 \inf_{\theta\in\Theta_0}\min_{1\leq j<M}\beta_j(\theta)
 \geq\gamma>0.
\]
For $M=1$ this condition is vacuous. Then there is a finite $C_M$ such that
\[
 \sup_{\theta\in\Theta_0}\max_{1\leq m\leq M}
       \|\partial_\theta V_m(\theta)\|_2\leq C_M.
\]
If $\sigma(A(\theta))\subseteq[a,b]$ throughout $\Theta_0$, then
\[
 \sup_{\theta\in\Theta_0}\max_{1\leq m\leq M}
       \|(\partial_\theta V_m)^\top r_m\|_2
 \leq\frac{b-a}{2}C_M.
\]
\end{lemma}
\begin{proof}
The Lanczos recurrence consists of smooth operations except for
normalization at a zero residual. For $w\neq0$,
\[
 D\!\left(\frac{w}{\|w\|_2}\right)[h]
 =\frac{1}{\|w\|_2}
    \left(I-\frac{ww^\top}{\|w\|_2^2}\right)h.
\]
The lower bound on the intermediate residual norms therefore makes
$V_m$ continuously differentiable for all $m\leq M$. Compactness and
the finite number of depths give the asserted $C_M$.
For $c_0=(a+b)/2$,
\[
 r_m=(I-V_mV_m^\top)(A-c_0I)v_m,
 \qquad \|r_m\|_2\leq\|A-c_0I\|_2\leq\frac{b-a}{2}.
\]
Multiplication by $\|\partial_\theta V_m\|_2$ gives the second bound.
If $r_M=0$, the corresponding residual-coupled sensitivity is zero.
\end{proof}

The constant $C_M$ may depend on $M$. For fixed finite $n$,
exact-arithmetic Lanczos terminates by step $n$; asymptotic convergence
as $m\to\infty$ therefore concerns a family of problems. Along such
a family, the error bound tends to zero whenever
\[
 \|u\|_2^2\|\xi_m\|_2\mathcal E_{m-1}(f';\mathcal I)\to0.
\]

\subsection{Example: the matrix exponential}
\label{app:exponential_example}
For $f(x)=e^{\tau x}$, $\tau\in\mathbb R$, let
$c_0=(a+b)/2$ and $d=(b-a)/2$. The Taylor polynomial of $f'$ about
$c_0$ of degree at most $m-1$ gives
\begin{equation}
 \mathcal E_{m-1}(f';\mathcal I)
 \leq |\tau|e^{\max_{x\in\mathcal I}\tau x}
                  \frac{(|\tau|d)^m}{m!}.
 \label{eq:app_exponential_poly_bound}
\end{equation}
Indeed, the $m$th derivative of $f'$ is $\tau^{m+1}e^{\tau x}$,
and $|x-c_0|\leq d$. Thus
\begin{equation}
 |g_{\mathrm{exact}}-g_{\mathrm{approx}}|
 \leq 2\|u\|_2^2\|\xi_m\|_2
       |\tau|e^{\max_{x\in\mathcal I}\tau x}
       \frac{(|\tau|d)^m}{m!}.
 \label{eq:app_exponential_gradient_bound}
\end{equation}
For $\tau=1$, the function-dependent factor is $e^bd^m/m!$.
For uniformly bounded $\|u\|_2$, its factorial decay yields convergence
if the basis-sensitivity factor is bounded, or more generally grows
slowly enough that their product tends to zero.

\subsection{Complex Hermitian matrices}
\label{app:hermitian_extension}
For Hermitian $A$, fixed complex $u$, real $\theta$, and real-valued $f$,
replace transposes by conjugate transposes. Then $G$ is Hermitian,
$S=V^*\partial_\theta V$ is skew-Hermitian, and the same proof gives
\[
 g_{\mathrm{exact}}-g_{\mathrm{approx}}
 =2\operatorname{Re}\!\left(e_m^*G(\partial_\theta V)^*r_m\right).
\]
The norm bound is unchanged. Complex-valued functions can be treated
by applying the real-valued result separately to their real and
imaginary parts.

\section{Additional experimental results}
\label{app:Extra_exp_results}

\subsection{Log-determinant optimization in GP training}
\label{app:Gpy_logdet}
We compare projected Lanczos gradients with GPyTorch's built-in
log-determinant gradient in Gaussian process regression. For $n$ training
observations, the negative log marginal likelihood per observation is
\[
 \mathcal L(\theta)=\frac{1}{2n}(y-\mu_\theta)^\top
 K_\theta^{-1}(y-\mu_\theta)
 +\frac{1}{2n}\log\det K_\theta+\frac12\log(2\pi),
\]
where $\mu_\theta$ is a constant mean vector and
\[
 K_\theta=\sigma_f^2 K_{\nu=3/2}(X,X;\ell)+\sigma_n^2I.
\]
The Mat\'ern-$3/2$ kernel has a separate lengthscale for each input
coordinate. Both methods use GPyTorch's solver and differentiation rule
for the inverse-quadratic term. For the log determinant, the baseline
uses GPyTorch's preconditioned stochastic estimator and solve-based
backward rule; our method uses the estimator and projected gradient
described below.

\paragraph{Preconditioned estimator.}
The rank-$r$ pivoted-Cholesky preconditioner has the form
\[
 P_\theta=R_\theta R_\theta^\top+\delta_\theta I,
 \qquad R_\theta\in\mathbb R^{n\times r},\quad \delta_\theta>0.
\]
The action of $P_\theta^{-1/2}$ is evaluated using the low-rank factor
and the eigendecomposition of $R_\theta^\top R_\theta$. Define
$B_\theta=P_\theta^{-1/2}K_\theta P_\theta^{-1/2}$. Then
\[
 \log\det K_\theta=\log\det P_\theta+\operatorname{tr}(\log B_\theta).
\]
At each iteration we draw independent $g_s\sim\mathcal N(0,I)$ and set
$z_s=\sqrt n\,g_s/\|g_s\|_2$, so that
$\mathbb E[z_sz_s^\top]=I$ and $\|z_s\|_2^2=n$.
Let $T_s$ be the $m$-step Lanczos tridiagonal obtained from $B_\theta$
with starting vector $g_s/\|g_s\|_2$. Our estimator is
\begin{equation}
 \widehat{\log\det K_\theta}
 =\log\det P_\theta+
  \frac{n}{p}\sum_{s=1}^p e_1^\top\log(T_s)e_1.
 \label{eq:gp-preconditioned-logdet}
\end{equation}
The Gaussian draws are resampled at iteration $k$ using seed $12345+k$
and are held fixed when computing that iteration's gradient.

Let $G_s$ be the projected sensitivity of $ne_1^\top\log(T_s)e_1$.
For parameter coordinate $\theta_j$, the projected log-determinant
gradient is
\[
 \partial_{\theta_j}\log\det P_\theta+
 \frac1p\sum_{s=1}^p
 \operatorname{tr}\!\left(
 V_sG_sV_s^\top\partial_{\theta_j}B_\theta\right).
\]
The bases $V_s$ and sensitivities $G_s$ are held fixed in these
contractions, while differentiation passes through $K_\theta$,
$P_\theta^{-1/2}$, and $\log\det P_\theta$. Pivot indices are held
fixed in the pivoted-Cholesky backward pass; the selected-row Cholesky
computations are differentiated.

\paragraph{Protocol and results.}
We use $4096$ examples from the UCI Protein
dataset~\citep{rana2013protein}, with nine input features and a seeded
split into $3276$ training and $820$ test examples. Features are
standardized and targets centered and scaled before subsampling and
the train--test split. Both methods use the same initial model
parameters, $120$ Adam updates with learning rate $0.025$, CPU double
precision, $p=20$ probes, $m=40$, and a rank-$15$ preconditioner.
Test RMSE is evaluated every three iterations.
% VERIFY GP_METADATA: the report does not list the complete initial ARD
% lengthscale vector, subset/split seeds, solver tolerances, or run versions.
% Preprocessing used the full data before the split; do not describe it as
% train-only preprocessing or silently substitute results from a new protocol.

Figure~\ref{fig:gp-protein-training} shows the estimated training
objective and test RMSE for one run. The objective curves use the
methods' respective stochastic estimators. GPyTorch's test RMSE initially
decreases and then rises, while the projected-gradient run continues
to improve. The final test RMSEs are approximately $0.745$ and $0.689$,
respectively.

At initialization, Table~\ref{tab:gp-first-iter-logdet} compares the
log-determinant estimates and their gradients with a dense Cholesky
reference. The diagnostic supplies the same raw Gaussian array and
initial parameters to both methods. Their effective probes differ:
GPyTorch transforms probes through its preconditioner, whereas our
estimator uses the normalized directions in $B_\theta$.
The errors include stochastic trace-estimation and numerical
approximation error. For the log determinant, the absolute discrepancy
is divided by $\max\{1,|\log\det K_\theta|\}$. Gradient errors are
Euclidean relative errors in the raw model-parameter coordinates,
normalized by the dense-reference gradient norm. The projected method
has the smaller gradient discrepancy in this diagnostic.

\begin{table}[t]
\centering
\small
\caption{Log-determinant diagnostic at initialization on Protein, using
$m=40$, $p=20$, and a rank-$15$ preconditioner. Errors use a dense
Cholesky reference and include stochastic estimation error. Gradient
errors are measured in the raw model-parameter coordinates.}
\label{tab:gp-first-iter-logdet}
\begin{tabular}{lcc}
\toprule
Method & Scaled logdet error & Rel. gradient error \\
\midrule
GPyTorch & $2.29\times10^{-3}$ & $1.44\times10^{-1}$ \\
Ours & $1.94\times10^{-3}$ & $5.10\times10^{-3}$ \\
\bottomrule
\end{tabular}
\end{table}

\begin{figure}[t]
 \centering
 \includegraphics[width=\linewidth]{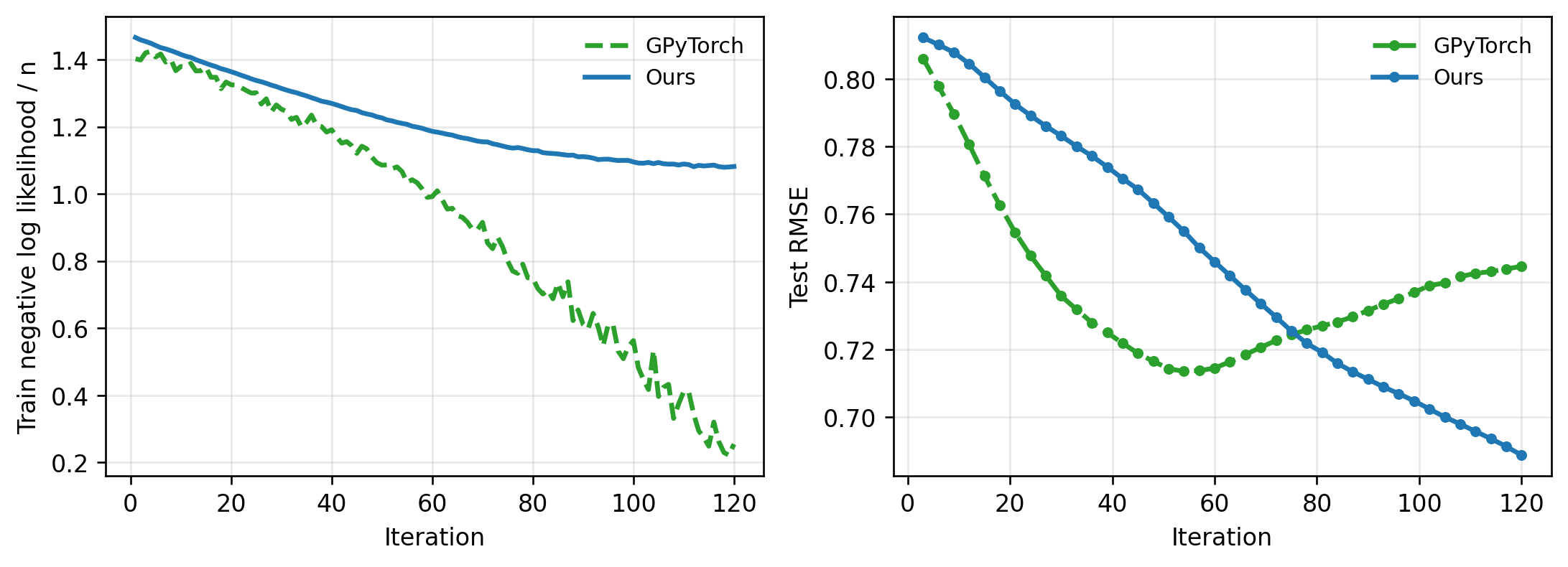}
 \caption{GP training on Protein with a Mat\'ern-$3/2$ kernel,
 $3276$ training and $820$ test examples, $m=40$, $p=20$, and a
 rank-$15$ preconditioner. Left: negative log marginal likelihood
 per training observation, using each method's stochastic estimator.
 Right: test RMSE. Both methods use the same preprocessing and split.}
 \label{fig:gp-protein-training}
\end{figure}

\subsection{Graph-sensitivity results}
\label{app:network-sensitivity}
For a symmetric adjacency matrix $A$, define total network
communicability and subgraph centrality by
\[
 \mathrm{TN}(A)=\mathbf1^\top e^A\mathbf1,
 \qquad
 \mathrm{SC}_{\ell}(A)=e_\ell^\top e^Ae_\ell.
\]
Both have the form $J_u(A)=u^\top e^Au$, with $u=\mathbf1$ for TN
and $u=e_\ell$ for SC. Their entrywise sensitivities are
\begin{equation}
 S_{ij}(A;u)=u^\top L_{\exp}(A,e_ie_j^\top)u
 =e_i^\top L_{\exp}(A^\top,uu^\top)e_j.
 \label{eq:network-entry-sensitivity}
\end{equation}
Here an entrywise sensitivity refers to the ambient matrix derivative.
For a symmetric perturbation of an off-diagonal undirected edge,
$E=e_ie_j^\top+e_je_i^\top$, the derivative is $S_{ij}+S_{ji}=2S_{ij}$.
The reported benchmarks use the entrywise convention.

\paragraph{Polarized projected approximation.}
The benchmark computes the bilinear derivative
$e_i^\top L_{\exp}(A,E)e_j$ with $E=uu^\top$. Since $A$ and $E$
are symmetric, polarization gives
\[
 e_i^\top e^Ae_j
 =\frac14\left(z_+^\top e^Az_+-z_-^\top e^Az_-\right),
 \qquad z_\pm=e_i\pm e_j.
\]
For each sign $\sigma\in\{+,-\}$, let $V_\sigma,T_\sigma$ be the
$m$-step Lanczos basis and tridiagonal generated from
$z_\sigma/\sqrt2$, and let $G_\sigma$ be the sensitivity of
$2e_1^\top e^{T_\sigma}e_1$. The implemented approximations are
\begin{align}
 \widehat v_{ij}
 &=\frac14\left(2e_1^\top e^{T_+}e_1
                    -2e_1^\top e^{T_-}e_1\right),\nonumber\\
 \widetilde S_{ij}
 &=\frac14\left[
 \operatorname{tr}(G_+V_+^\top EV_+)
 -\operatorname{tr}(G_-V_-^\top EV_-)\right].
 \label{eq:network-polarized-approximation}
\end{align}
Here $E=\mathbf1\mathbf1^\top$ for TN and $E=e_\ell e_\ell^\top$
for SC. The two independent scalar Lanczos recurrences are evaluated
in a two-column batch, using float64 arithmetic without
reorthogonalization and a breakdown tolerance of $10^{-14}$.
Each $G_\sigma$ includes the input norm squared $\|z_\sigma\|_2^2=2$.

\paragraph{Graphs and reference.}

We compare with the modified-Arnoldi method of \citet{kressner2026}.
The reference applies \texttt{scipy.sparse.linalg.expm\_multiply} to
the $2n$-dimensional Fr\'echet block operator using the library-default
tolerances.

The loaded \texttt{ca-HepTh} matrix has $9877$ rows and $51{,}946$
nonzeros; \texttt{com-Amazon} has $334{,}863$ rows and $1{,}851{,}744$
nonzeros. The loader converts the node IDs to contiguous zero-based
row indices, sums duplicate entries, removes self-loops, and
symmetrizes the adjacency matrix as $(A+A^\top)/2$. No spectral
rescaling is applied. For \texttt{ca-HepTh}, $(i,j)=(7200,6969)$ and
$\ell=7200$; for \texttt{com-Amazon}, $(i,j)=(0,53525)$ and $\ell=0$.
The selected entries have weights $0$ and $0.5$, respectively, under
this loading convention.

Figure~\ref{fig:network-sensitivity-grad-error} shows that both methods
give highly accurate sensitivities at $m=32$, with several results
close to machine precision. At smaller depths, accuracy depends on
the graph and objective; modified Arnoldi is more accurate on
\texttt{ca-HepTh}/SC.

The separate value diagnostic in
Figure~\ref{fig:network-sensitivity-value-error} measures
$v=e_i^\top e^Ae_j$, with scaled error
\[
 \varepsilon_v=\frac{|\widehat v-v|}{\max\{1,|v|\}}.
\]
The TN and SC labels identify the corresponding sensitivity
directions $uu^\top$ in \eqref{eq:network-entry-sensitivity}; the
value target is the selected exponential entry in both cases.
Thus this figure does not measure errors in the TN and SC objective
values themselves.
\paragraph{Timing protocol.}
CPU measurements use a MacBook Pro with an Apple M4 Pro chip,
Python 3.11.14, NumPy 2.4.4, and SciPy 1.17.1 on macOS 26.2.
The runner requests one thread through the thread-pool settings.
Outer timers cover complete value-and-gradient calls for both methods.
For the projected method this includes matrix preparation, both
Lanczos recurrences, tridiagonal construction, small matrix
exponentials and sensitivities, direction contractions, and scalar
assembly. The unused exact Lanczos adjoint is not computed.
The KO timer includes its complete call, including preparation and
the start-orientation check. Graph loading and reference calculations
are excluded.

Each method receives one warm-up call per graph, objective, and depth,
followed by three measured calls. The method order alternates between
repetitions. For paired times $t_{\mathrm{KO},r}$ and
$t_{\mathrm{ours},r}$, we report the mean of
$s_r=t_{\mathrm{KO},r}/t_{\mathrm{ours},r}$.
The $95\%$ intervals are
$\bar s\pm t_{0.975,2}\operatorname{sd}(s_1,s_2,s_3)/\sqrt3$,
where the sample standard deviation uses two degrees of freedom.
The timing runs are separate from the recorded accuracy curves;
their numerical outputs agree to rounding accuracy.

Table~\ref{tab:graph-complete-timings} reports mean evaluation times
and paired speedups. The projected implementation is faster in all
eight tested CPU configurations, with mean speedups of
$1.37$--$2.69\times$. Figure~\ref{fig:network-sensitivity-speedup}
shows the same speedup statistic and intervals.

\begin{table}[t]
  \centering
  \small
  \begin{tabular}{llrrrr}
    \toprule
    Graph & Obj. & $m$ & Ours (ms) & KO (ms) & Speedup \\
    \midrule
    HepTh & TN & 16 & 5.11 & 7.01 & $1.37 \pm 0.13$ \\
    HepTh & TN & 32 & 9.81 & 14.14 & $1.44 \pm 0.16$ \\
    HepTh & SC & 16 & 5.06 & 7.09 & $1.40 \pm 0.07$ \\
    HepTh & SC & 32 & 9.95 & 14.65 & $1.47 \pm 0.03$ \\
    Amazon & TN & 16 & 165.02 & 375.70 & $2.28 \pm 0.03$ \\
    Amazon & TN & 32 & 329.80 & 884.02 & $2.68 \pm 0.03$ \\
    Amazon & SC & 16 & 169.88 & 371.73 & $2.19 \pm 0.14$ \\
    Amazon & SC & 32 & 332.19 & 892.83 & $2.69 \pm 0.09$ \\
    \bottomrule
  \end{tabular}
  \caption{Complete CPU value-and-gradient timings on the Apple M4 Pro.
  Times are means of three measured calls after one warm-up per method
  and case. KO denotes modified Arnoldi. Speedup is the mean paired
  ratio $t_{\mathrm{KO}}/t_{\mathrm{ours}}$; the $\pm$ values are
  half-widths of $95\%$ Student-$t$ intervals.}
  \label{tab:graph-complete-timings}
\end{table}

\begin{figure}[t]
 \centering
 \includegraphics[width=0.9\linewidth]{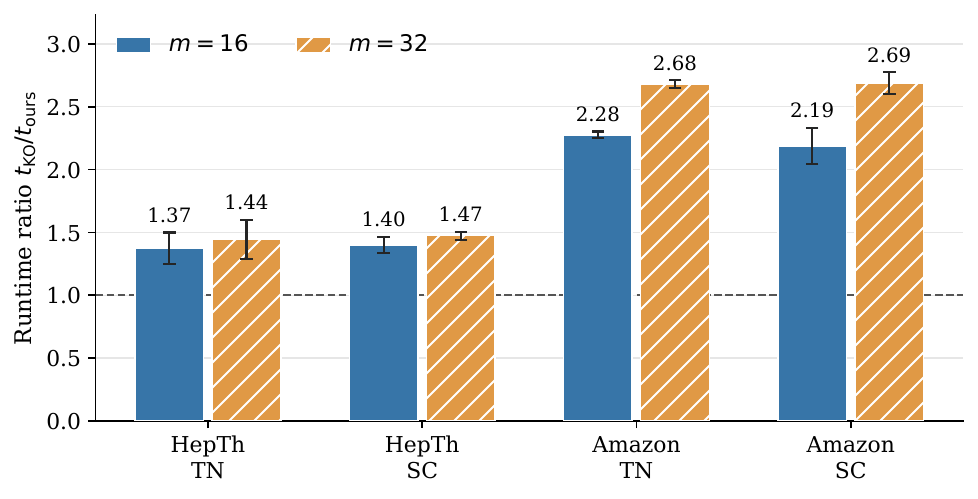}
 \caption{Mean paired runtime ratios of modified Arnoldi to projected
 Lanczos for complete value-and-gradient evaluations on the Apple M4 Pro.
 Error bars show $95\%$ Student-$t$ intervals from three paired
 measurements. The dashed line denotes equal runtime.}
 \label{fig:network-sensitivity-speedup}
\end{figure}

\begin{figure}[t]
 \centering
 \includegraphics[width=0.85\linewidth]{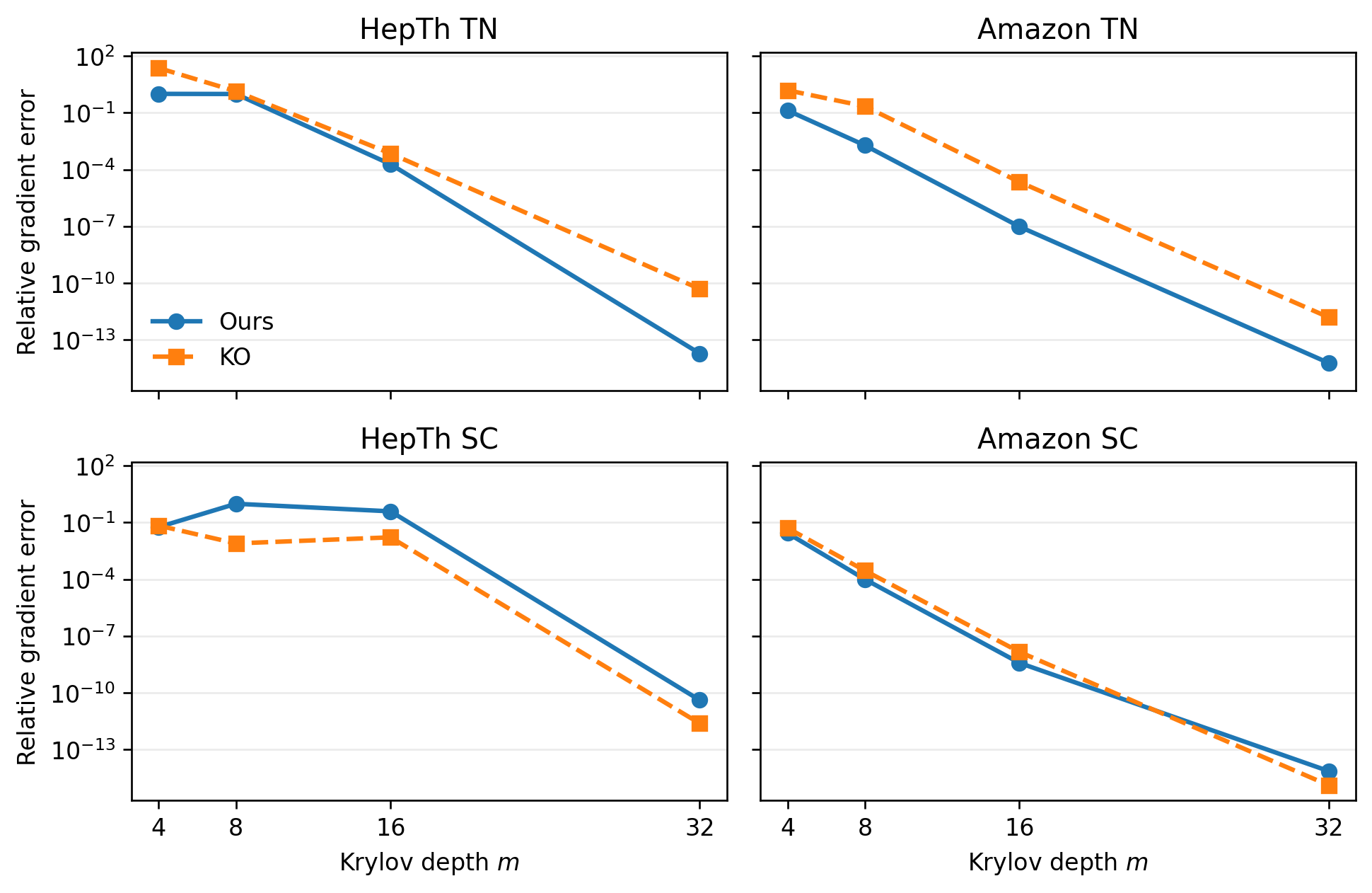}
 \caption{Relative entrywise-sensitivity error versus Krylov depth $m$
 for TN and SC on \texttt{ca-HepTh} and \texttt{com-Amazon}.}
 \label{fig:network-sensitivity-grad-error}
\end{figure}
\begin{figure}[t]
 \centering
  \includegraphics[width=0.85\linewidth]{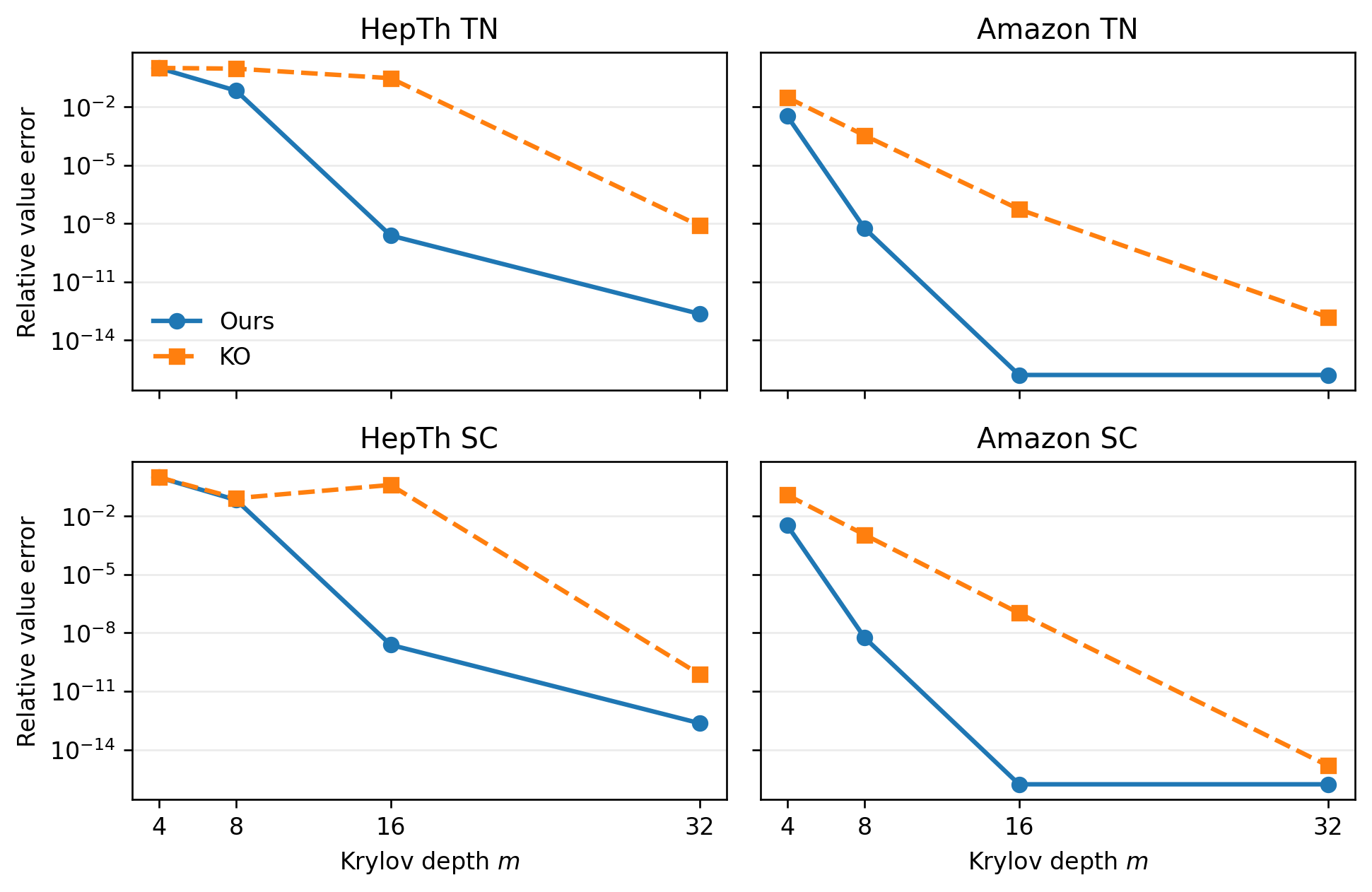}
 \caption{Scaled error in the selected matrix-exponential entry
 $e_i^\top e^Ae_j$ versus Krylov depth $m$. The normalization is
 $\max\{1,|e_i^\top e^Ae_j|\}$. TN and SC identify the associated
 sensitivity configurations.}
 \label{fig:network-sensitivity-value-error}
\end{figure}

\subsection{Hamiltonian-learning setup and results}
\label{app:hamiltonian-learning}

\paragraph{Hamiltonian parameterization.}
We consider an $L$-site spin system with Hamiltonian
\[
    H_\theta = \sum_{j=1}^{p} \theta_j P_j,
    \qquad
    \theta \in \mathbb{R}^{p},
    \qquad
    p=3L-1.
\]
Let $I$, $X$, and $Z$ denote the $2\times 2$ identity,
Pauli-$X$, and Pauli-$Z$ matrices. The single-site operators are
\[
    X_i = I^{\otimes i}\otimes X\otimes I^{\otimes(L-i-1)},
    \qquad
    Z_i = I^{\otimes i}\otimes Z\otimes I^{\otimes(L-i-1)},
\]
for $i=0,\ldots,L-1$, and the nearest-neighbor interactions are
\[
    Z_iZ_{i+1}
    =
    I^{\otimes i}\otimes Z\otimes Z\otimes I^{\otimes(L-i-2)},
    \qquad i=0,\ldots,L-2.
\]
The operator dictionary is
\[
    \{P_j\}_{j=1}^{p}
    =
    \{X_i\}_{i=0}^{L-1}
    \cup
    \{Z_i\}_{i=0}^{L-1}
    \cup
    \{Z_iZ_{i+1}\}_{i=0}^{L-2}.
\]
Thus, the model includes transverse and longitudinal single-site
fields and nearest-neighbor Ising interactions with open boundaries.
We use $L=8$, giving a Hilbert-space dimension of $256$ and
$23$ unknown parameters.

\paragraph{Data generation and optimization.}
We construct a ground-truth parameter vector $\theta_\star$ by
setting the coefficients of the $X_i$ and $Z_i$ terms to $0.35$
and those of the $Z_iZ_{i+1}$ terms to $1.0$, then adding independent
Gaussian perturbations with standard deviation $0.05$ to all
coefficients. For each of $30$ training examples, we sample a
real Gaussian vector and normalize it to obtain $x_j\in\mathbb R^{256}$,
then represent it in complex double precision. We independently draw
an evolution time $t_j\sim\mathrm{Unif}(2,6)$. The target state is
\[
    c_j = \exp(-\mathrm{i}t_jH_{\theta_\star})x_j.
\]
The experiment uses noiseless full-state targets. We recover
the parameters by minimizing
\[
    \mathcal{L}(\theta)
    =
    \sum_{j=1}^{30}
    \left\|c_j-\exp(-\mathrm{i}t_jH_\theta)x_j\right\|_2^2.
\]
Optimization starts from
\[
    \theta_0 = \theta_\star + 0.1\,\xi,
    \qquad
    \xi\sim\mathcal{N}(0,I_p),
\]
and uses $400$ Adam steps with learning rate $0.01$,
$(\beta_1,\beta_2)=(0.9,0.999)$, and $\epsilon=10^{-8}$.
All methods start from the same parameters. We compare Krylov depths
$m\in\{30,40,50\}$ with a dense Fr\'echet-derivative baseline.
Both methods evaluate the squared-residual loss using dense
\texttt{scipy.linalg.expm} propagation at every step. The baseline
gradient uses \texttt{scipy.linalg.expm\_frechet}; our method replaces
the gradient computation by projected quadratic-form sensitivities.

\paragraph{Gradient from quadratic forms.}
Here $H_\theta$ is real symmetric and
$U_j(\theta)=\exp(-\mathrm{i}t_jH_\theta)$ is unitary. Since the initial
and target states have unit norm,
\[
 \mathcal L(\theta)
 =\sum_{j=1}^{30}\left(2-2\operatorname{Re}
       \bigl(c_j^*U_j(\theta)x_j\bigr)\right).
\]
The exact parameter derivative is
\[
 \partial_{\theta_k}\mathcal L
 =-2\sum_{j=1}^{30}\operatorname{Re}
       \bigl(c_j^*L_{f_j}(H_\theta,P_k)x_j\bigr),
 \qquad f_j(x)=e^{-\mathrm{i}t_jx}.
\]
To approximate it, the implementation uses polarization. For any
complex matrix $M$ and $q_M(v)=v^*Mv$,
\begin{equation}
 a^*Mb=\frac14\sum_{\ell=0}^{3}\mathrm{i}^{-\ell}
           q_M(a+\mathrm{i}^{\ell}b).
 \label{eq:hamiltonian-polarization}
\end{equation}
Write $z_{j\ell}=c_j+\mathrm{i}^{\ell}x_j=a_{j\ell}+\mathrm{i}b_{j\ell}$,
with $a_{j\ell},b_{j\ell}$ real. For real-valued $h$, the matrix
$h(H_\theta)$ is real symmetric, so
\[
 z_{j\ell}^*h(H_\theta)z_{j\ell}
 =a_{j\ell}^\top h(H_\theta)a_{j\ell}
  +b_{j\ell}^\top h(H_\theta)b_{j\ell}.
\]
For a real starting vector $v$, denote the projected sensitivity by
\[
 D_{m,h}(v;P_k)
 =\operatorname{tr}(V_{m,v}G_{m,h,v}V_{m,v}^\top P_k),
\]
where $G_{m,h,v}$ includes the factor $\|v\|_2^2$ as in
\eqref{eq:method_G}. For $v=0$, set $D_{m,h}(v;P_k)=0$.
The approximation to the derivative of $c_j^*h(H_\theta)x_j$ is
\[
 \widehat b_{j,h,k}
 =\frac14\sum_{\ell=0}^{3}\mathrm{i}^{-\ell}
 \bigl(D_{m,h}(a_{j\ell};P_k)+D_{m,h}(b_{j\ell};P_k)\bigr).
\]
With $h_{j,c}(x)=\cos(t_jx)$ and $h_{j,s}(x)=\sin(t_jx)$, the
implemented loss gradient is
\begin{equation}
 \widetilde{\partial_{\theta_k}\mathcal L}
 =-2\sum_{j=1}^{30}
 \left(\operatorname{Re}\widehat b_{j,h_{j,c},k}
       +\operatorname{Im}\widehat b_{j,h_{j,s},k}\right).
 \label{eq:hamiltonian-projected-gradient}
\end{equation}
The real-symmetric quadratic-form rule therefore applies to each
term. The implementation builds separate bases for the real and
imaginary starts and for the cosine and sine functions: $16$ basis
builds per sample, including repeated starts. Each resulting
sensitivity is contracted with all $23$ parameter matrices.
Lanczos uses float64 arithmetic without reorthogonalization and a
breakdown threshold of $10^{-14}$. The loss and gradient are summed
over the $30$ samples without averaging.

\paragraph{Error measures and timing.}
The initial gradient error in Table~\ref{tab:hamiltonian-learning} is
\[
 \frac{\|g_m(\theta_0)-g_{\mathrm{dense}}(\theta_0)\|_2}
 {\max\{\|g_{\mathrm{dense}}(\theta_0)\|_2,10^{-14}\}},
\]
and parameter recovery error is the absolute norm
$\|\theta-\theta_\star\|_2$. Times cover loss-and-gradient evaluations,
including dense propagation in both methods, and exclude Adam updates.
Speedups compare the accumulated evaluation times over one $400$-step
history per method; no independent timing repetitions are used.
% VERIFY HAMILTONIAN_SEEDS: the report confirms separate data and
% initialization seeds but does not provide their values.

\paragraph{Results.}
Figure~\ref{fig:hamiltonian-learning-step} shows the training loss
and parameter recovery error as functions of optimization step.
For $m=30$, optimization initially follows the dense trajectory
but then plateaus, consistent with the larger gradient
approximation error. Increasing the depth to $m=40$ allows the
method to track the dense baseline for longer and attain a
smaller final parameter error. At $m=50$, the Krylov run closely
tracks the dense trajectories over much of the optimization,
although the dense baseline ultimately attains a lower error.

These trajectories complement
Table~\ref{tab:hamiltonian-learning}: the final parameter error
decreases from $5.43\times10^{-3}$ at $m=30$ to
$1.05\times10^{-7}$ at $m=50$, compared with
$2.00\times10^{-10}$ for the dense baseline.
The reported speedups range from $14.8\times$ at $m=30$ to
$10.4\times$ at $m=50$ relative to the dense baseline.

\begin{figure}[t]
\centering
\begin{subfigure}[t]{0.48\linewidth}
  \centering
  \includegraphics[width=\linewidth]{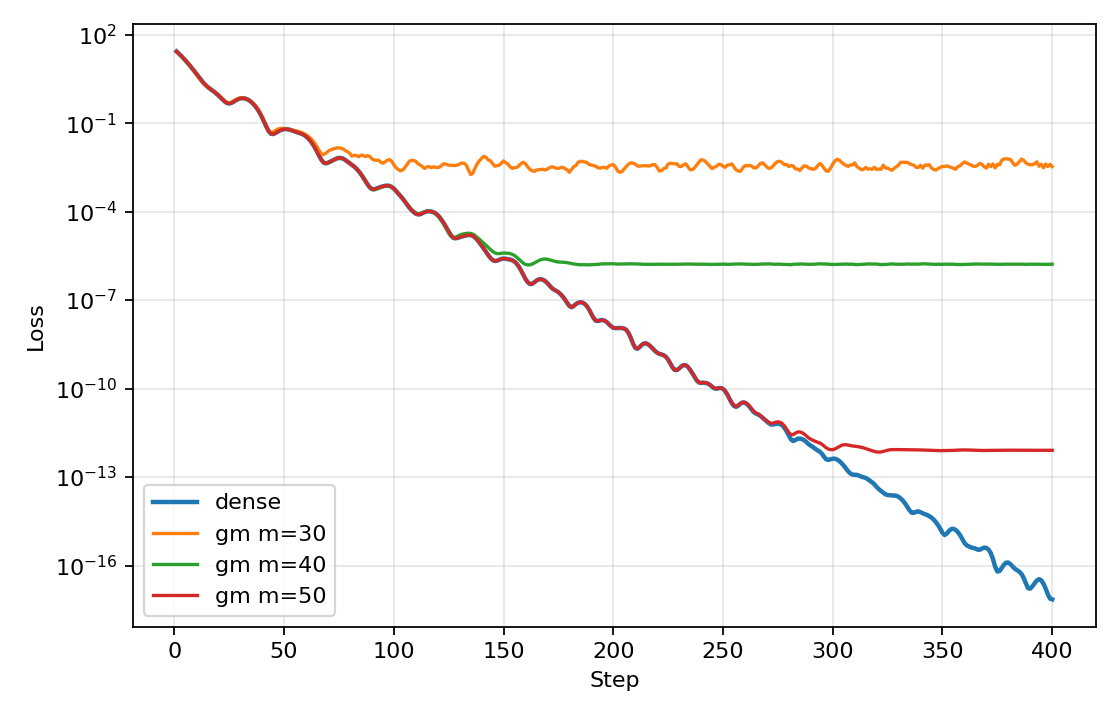}
  \caption{Training loss vs. optimization step.}
\end{subfigure}
\hfill
\begin{subfigure}[t]{0.48\linewidth}
  \centering
  \includegraphics[width=\linewidth]{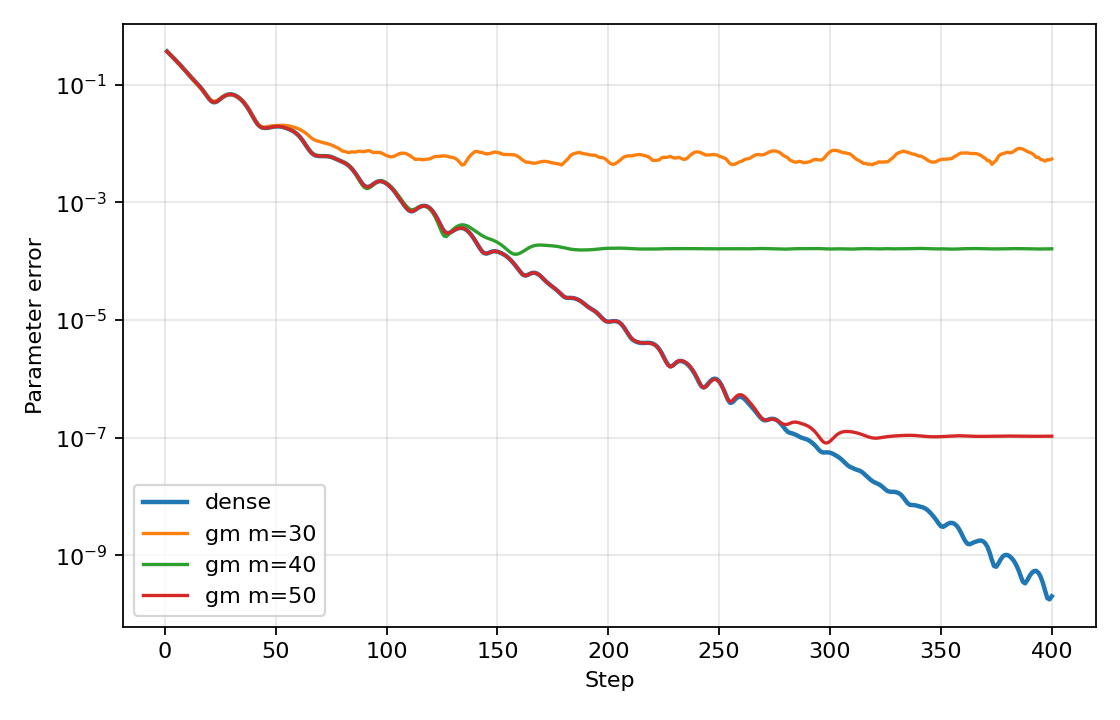}
  \caption{Absolute parameter error vs. optimization step.}
\end{subfigure}
\caption{
Hamiltonian-learning convergence versus optimization step. The forward
loss is evaluated by dense propagation in every run. Larger Krylov
depths give approximate-gradient trajectories closer to the dense
Fr\'echet-derivative baseline.
}
\label{fig:hamiltonian-learning-step}
\end{figure}

\subsection{Gradient stability for an affine fractional-covariance model}
\label{app:lanczos-regimes}

We examine the stability of Lanczos-based gradients as the Krylov depth
increases in a synthetic fractional-covariance model where the operator admits
a known eigendecomposition, allowing matrix-free approximations to be
compared with spectral references. Its affine parameterization also
leaves the exact-arithmetic Lanczos basis unchanged under shifts and
positive scalings. This setting isolates the numerical behavior of the
gradient computations from the basis-variation error analyzed in the
main text.

\paragraph{Model and objective.}
Let $H$ be the normalized Walsh--Hadamard transform~\cite{walsh1923closed} and define
$B=H\operatorname{diag}(\lambda)H^\top$. We consider
\begin{equation}
    A_\theta=\mu I+sB,
    \qquad
    K_\theta=A_\theta^{-\alpha},
    \qquad
    \theta=(\log\mu,\log s,\alpha).
    \label{eq:regimes-model}
\end{equation}
The dimension is $n=32{,}768$. The spectrum comprises 128 logarithmically
spaced cluster centers in $[10^{-10},1]$, each repeated 256 times,
perturbed multiplicatively by $\exp(10^{-8}z_i)$ with
$z_i\sim\mathcal N(0,1)$, and randomly permuted. Spectrum generation
uses a CPU PyTorch generator with seed $17$ and float32 arithmetic, so
the smallest perturbations may be lost to rounding. Matrix--vector
products use the fast Walsh--Hadamard transform without forming $B$.

We generate $N_s=2$ independent observations
$y_r\sim\mathcal N(0,K_{\theta_\star})$ with
$(\mu_\star,s_\star,\alpha_\star)=(10^{-6},1,0.5)$ and a separate
CPU PyTorch generator. Gaussian coefficients are drawn
and scaled in the spectral basis in float64, cast to float32, and
transformed by the CPU Hadamard routine. The regularized negative log
likelihood, up to an additive constant, is
\begin{equation}
    \begin{split}
        \mathcal L(\theta)
        &=D(\theta)-\frac{N_s\alpha}{2}\log\det A_\theta
        +\frac{\rho}{2}\|\theta-\theta_\star\|_2^2,
        \\
        D(\theta)
        &=\frac12\sum_{r=1}^{N_s}y_r^\top A_\theta^\alpha y_r,
        \qquad \rho=10^{-4},
    \end{split}
    \label{eq:regimes-objective}
\end{equation}
where $\theta_\star=(\log\mu_\star,\log s_\star,\alpha_\star)$.
The penalty is centered at the generating parameters as part of this
controlled synthetic experiment. All methods start from
$(\mu_0,s_0,\alpha_0)=(10^{-5},5,1.5)$, for which
$\kappa(A_{\theta_0})\simeq5\times10^5$, and use eight Adam updates
with learning rate $0.02$, $(\beta_1,\beta_2)=(0.9,0.999)$, and
$\epsilon=10^{-8}$. The log determinant, regularizer, and their
derivatives are evaluated from spectral or analytic formulas for every
method; only the data term $D$ and its gradient are approximated.

\paragraph{Spectral reference.}
Writing $a_i=\mu+s\lambda_i$ and $\widetilde y_r=H^\top y_r$ gives
\begin{equation}
    D(\theta)=\frac12\sum_{r=1}^{N_s}\sum_{i=1}^n
        \widetilde y_{ri}^{\,2}a_i^\alpha,
    \qquad
    \log\det A_\theta=\sum_{i=1}^n\log a_i.
    \label{eq:regimes-reference}
\end{equation}
The derivatives of $D$ with respect to $\log\mu$, $\log s$, and
$\alpha$ follow by replacing $a_i^\alpha$ in the sum with
$\alpha\mu a_i^{\alpha-1}$, $\alpha s\lambda_i a_i^{\alpha-1}$,
and $a_i^\alpha\log a_i$, respectively. Reference calculations use CPU
float64 arithmetic and the same float32 spectrum and observations as
the GPU computations, promoted to float64; $H^\top y_r$ is recomputed
in float64. The reference evaluates the stored parameter coordinates
in float64, whereas GPU matrix--vector products evaluate
$\exp(\log\mu)$ and $\exp(\log s)$ in float32. Reported discrepancies
therefore include these rounding effects as well as approximation error.

\paragraph{Affine specialization of the projected gradient.}
For fixed $y_r$, $s>0$, and no intermediate breakdown, exact-arithmetic
Lanczos satisfies
\begin{equation}
    V_r(\mu I+sB,y_r)=V_r(B,y_r),
    \qquad
    T_r(\mu I+sB,y_r)=\mu I+sT_r(B,y_r).
    \label{eq:regimes-affine-invariance}
\end{equation}
Consequently, the basis-variation term vanishes in the $\log\mu$ and
$\log s$ directions. Let
$\widehat D_r=\tfrac12\|y_r\|_2^2e_1^\top T_r^\alpha e_1$ and
let $G_r$ be its sensitivity with respect to $T_r$ at fixed $\alpha$,
so that $d_T\widehat D_r=\operatorname{tr}(G_r^\top dT_r)$.
The projected data-gradient formulas reduce to
\begin{equation}
    \begin{split}
        \widetilde g^D_{\log\mu}
        &=\sum_{r=1}^{N_s}\mu\operatorname{tr}(G_r),
        \\
        \widetilde g^D_{\log s}
        &=\sum_{r=1}^{N_s}
        \operatorname{tr}\!\left(G_r^\top(T_r-\mu I)\right),
        \\
        \widetilde g^D_\alpha
        &=\frac12\sum_{r=1}^{N_s}\|y_r\|_2^2
        e_1^\top T_r^\alpha\log(T_r)e_1.
    \end{split}
    \label{eq:regimes-specialized-gradient}
\end{equation}
The factor $1/2$ is included in $G_r$, and the $\alpha$ component
accounts for the explicit dependence of $x^\alpha$ on the exponent.
Together, these expressions give the exact gradient of the finite
Lanczos quadrature in exact arithmetic.

We apply \eqref{eq:regimes-specialized-gradient} directly to the computed,
unreorthogonalized tridiagonal matrices and add the shared analytic
terms to obtain the objective gradient. This specialization requires
neither storage nor replay of the Lanczos basis. In finite precision it
uses the affine identities directly, without forming
$V_r^\top(\partial_\theta A_\theta)V_r$ from the computed basis.
The experiment thus evaluates gradient stability for this affine model;
it does not measure the effect of nonzero basis variation or the cost
of general operator VJPs.

\paragraph{Baselines and implementation.}
We compare the projected gradient with reverse-mode AD through Lanczos,
with and without full reorthogonalization, and a Chebyshev approximation
with an analytic reverse recurrence. Both AD baselines use a
Fr\'echet divided-difference rule for the terminal sensitivity with
respect to $T$, avoiding differentiation through the eigensolver.
No custom Lanczos-adjoint recurrence is used in this experiment.

For $f(x)=x^\alpha$, the terminal rule substitutes $f'(t_i)$ for the
$(i,j)$ divided difference whenever
$|t_i-t_j|\leq10^{-8}\max\{\max_\ell|t_\ell|,1\}$, where $t_i$
are the Ritz values. For distinct, nearly equal Ritz values, this
entrywise rule need not be symmetric. The dense sensitivity is formed
on the CPU in float64 and cast to float32 on MPS for recurrence
backpropagation or the affine contraction. The explicit exponent
derivative uses $\partial_\alpha x^\alpha=x^\alpha\log x$.

The Chebyshev approximation recomputes its spectral interval at each
evaluation as
\begin{equation}
    \begin{split}
        \ell&=\max\{(\mu+s\lambda_{\min})(1-10^{-8}),10^{-30}\},
        \\
        u&=\mu+s\lambda_{\max}(1+10^{-6}).
    \end{split}
    \label{eq:regimes-chebyshev-interval}
\end{equation}
Degree-$d$ coefficients are computed by cosine sums at
$\max(2d+1,64)$ Chebyshev--Gauss nodes in CPU float64 and cast to
float32. The two samples are processed in a batch. The reverse
recurrence retains the forward polynomial states and holds the
interval and coefficients fixed when differentiating with respect to
$\log\mu$ and $\log s$. A second coefficient series for
$x^\alpha\log x$, evaluated using the same states, approximates the
$\alpha$ derivative. These derivatives approximate the matrix-function
gradient; they exclude differentiation of the adaptive interval and
coefficient construction.

Chebyshev degree is $d=m$ for $m\in\{128,384,768\}$. The degree
used in the $m=1536$ comparison is not specified in the available run
record. The record also does not resolve the reorthogonalization sweep
count, Ritz-clipping setting, or breakdown threshold. The comparisons
use fixed configurations without matching error tolerances.

\paragraph{Error measures and timing.}
We evaluate the scaled errors
\begin{equation}
    \varepsilon_D=
    \frac{|\widehat D-D|}{\max(|D|,1)},
    \qquad
    \varepsilon_g=
    \frac{\|\widetilde{\nabla_\theta\mathcal L}
        -\nabla_\theta\mathcal L\|_2}
        {\max(\|\nabla_\theta\mathcal L\|_2,1)},
    \label{eq:regimes-errors}
\end{equation}
where $\widehat D=\sum_r\widehat D_r$ for the Lanczos methods.
References are evaluated at each method's current parameters, and all
gradients use the coordinates $(\log\mu,\log s,\alpha)$. Because the
remaining objective terms are evaluated identically, the numerator of
$\varepsilon_g$ equals the data-gradient discrepancy; the denominator
uses the full objective gradient.

Computations use an Apple M4 Pro laptop with 24 GB of unified memory,
macOS 26.2, and PyTorch MPS. Large-vector operations use float32;
small spectral computations and references use CPU float64. Timings
include evaluation of the approximate data quadratic and its gradient,
small-matrix processing, and host--device transfers. Reference
diagnostics, shared analytic terms, and optimizer updates are excluded.
Lanczos timings include explicit MPS synchronization before and after
each sample. Chebyshev timings use blocking CPU scalar reads to complete
the required GPU outputs, without explicit boundary synchronization.
Measurements use a single run without a separate warm-up; optimization
means cover all nine evaluations $k=0,\ldots,8$.

\begin{table}[t]
    \centering
    \caption{
        Value and gradient evaluation at the shared initialization, with
        Chebyshev degree $d=m$. Times include the approximate data term
        and its gradient; errors are defined in
        \eqref{eq:regimes-errors}. The full-reorthogonalized AD evaluation
        at $m=768$ was reported to exhaust memory.
    }
    \label{tab:regimes-single}
    \small
    \begin{tabular}{rlrrr}
        \toprule
        $m$ & Method & Time (s)
        & $\varepsilon_g$ & $\varepsilon_D$ \\
        \midrule
        384 & Full-reorth. AD
        & 9.065 & $4.29\times10^{-8}$ & $2.90\times10^{-7}$ \\
        384 & No-reorth. AD
        & 2.124 & $1.84\times10^{-4}$ & $2.86\times10^{-7}$ \\
        384 & Projected (ours)
        & 1.146 & $5.80\times10^{-8}$ & $2.86\times10^{-7}$ \\
        384 & Chebyshev
        & 0.899 & $2.26\times10^{-5}$ & $1.79\times10^{-4}$ \\
        \midrule
        768 & Full-reorth. AD
        & \multicolumn{3}{c}{Out of memory (reported)} \\
        768 & No-reorth. AD
        & 4.627 & $2.88\times10^{-2}$ & $2.89\times10^{-7}$ \\
        768 & Projected (ours)
        & 2.422 & $5.08\times10^{-8}$ & $2.89\times10^{-7}$ \\
        768 & Chebyshev
        & 1.793 & $1.47\times10^{-5}$ & $1.88\times10^{-4}$ \\
        \bottomrule
    \end{tabular}
\end{table}

\begin{figure}[t]
    \centering
    \includegraphics[width=\linewidth]{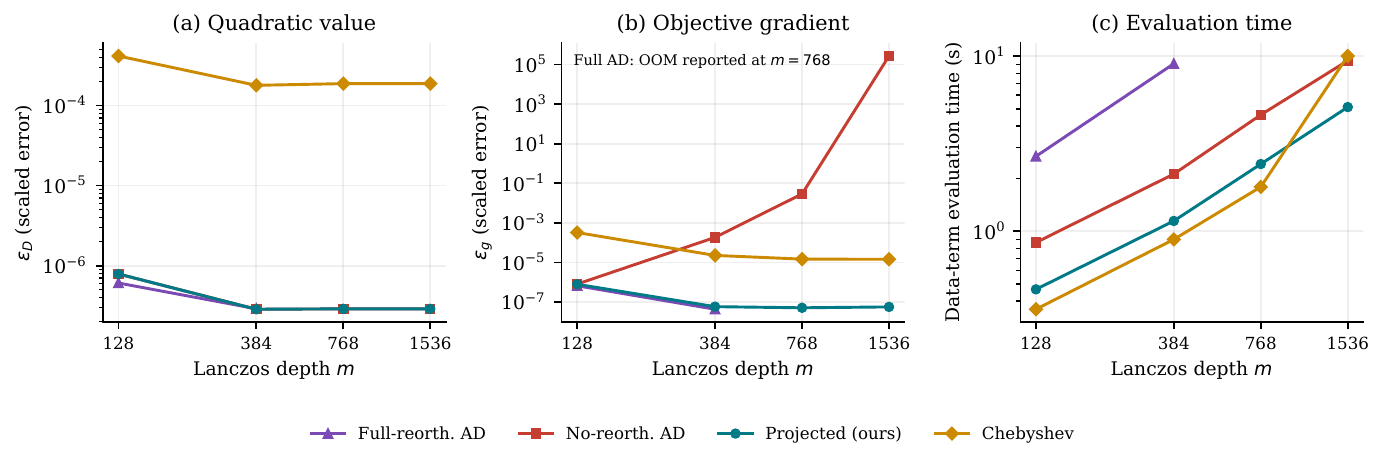}
    \caption{
        Value error, gradient error, and evaluation time as functions
        of Lanczos depth at the shared initialization. Errors use
        \eqref{eq:regimes-errors}; times include the data term and
        its gradient. Chebyshev degree equals $m$ at the first three
        depths and is unspecified at $m=1536$. No completed
        full-reorthogonalized AD measurement is available at
        $m=768$ or $1536$.
    }
    \label{fig:regimes-depth}
\end{figure}

\paragraph{Gradient accuracy with increasing depth.}
Figure~\ref{fig:regimes-depth} and Table~\ref{tab:regimes-single}
show a separation between forward-value and gradient accuracy for
unreorthogonalized AD. At $m=384$, its value error is
$2.86\times10^{-7}$, while its gradient error is
$1.84\times10^{-4}$. The projected and full-reorthogonalized AD
gradients have errors $5.80\times10^{-8}$ and $4.29\times10^{-8}$,
respectively, with the projected evaluation taking $1.146$ s compared
with $9.065$ s for full-reorthogonalized AD, a speedup of $7.91\times$.
At $m=768$, the unreorthogonalized AD gradient error increases to
$2.88\times10^{-2}$ despite a value error of $2.89\times10^{-7}$.
The projected gradient error remains $5.08\times10^{-8}$.
Chebyshev is faster at these degrees but has larger value and gradient
errors than the projected method.

\begin{table}[t]
    \centering
    \caption{
        Optimization over eight Adam updates. Panel (a) reports mean
        evaluation time and maximum errors over $k=0,\ldots,8$;
        panel (b) gives the final parameters, with mass reported as
        $10^5\mu$. Chebyshev degree is $128$ in the first block and
        is unspecified in the available record for the second block.
        Both blocks use the same spectral-reference trajectory.
    }
    \label{tab:regimes-fit}
    \small
    \setlength{\tabcolsep}{4pt}

    {\normalsize\textbf{(a) Timing and accuracy}}
    \par\smallskip

    \begin{tabular}{rlrrr}
        \toprule
        $m$ & Method & Mean (s)
        & Max $\varepsilon_g$ & Max $\varepsilon_D$ \\
        \midrule
        128 & Full-reorth. AD
        & 1.420 & $2.72\times10^{-6}$ & $4.06\times10^{-6}$ \\
        128 & No-reorth. AD
        & 0.852 & $5.10\times10^{-6}$ & $5.21\times10^{-6}$ \\
        128 & Projected (ours)
        & 0.474 & $3.75\times10^{-6}$ & $5.73\times10^{-6}$ \\
        128 & Chebyshev
        & 0.349 & $8.53\times10^{-4}$ & $1.99\times10^{-3}$ \\
        \midrule
        1536 & No-reorth. AD
        & 9.146 & $2.72\times10^{5}$ & $4.04\times10^{-7}$ \\
        1536 & Projected (ours)
        & 5.201 & $1.82\times10^{-7}$ & $4.39\times10^{-7}$ \\
        1536 & Chebyshev
        & 9.892 & $7.86\times10^{-5}$ & $8.85\times10^{-4}$ \\
        \bottomrule
    \end{tabular}

    \par\medskip
    {\normalsize\textbf{(b) Final parameter values}}
    \par\smallskip

    \begin{tabular}{rlrrr}
        \toprule
        $m$ & Method & $10^5\mu$ & $s$ & $\alpha$ \\
        \midrule
        128 & Full-reorth. AD
        & 1.172 & 5.528 & 1.340 \\
        128 & No-reorth. AD
        & 1.172 & 5.528 & 1.340 \\
        128 & Projected (ours)
        & 1.172 & 5.528 & 1.340 \\
        128 & Chebyshev
        & 1.172& 5.528 & 1.340 \\
        \midrule
        1536 & No-reorth. AD
        & 1.088 & 4.635 & 1.340 \\
        1536 & Projected (ours)
        & 1.172 & 5.528 & 1.340 \\
        1536 & Chebyshev
        & 1.172 & 5.528 & 1.340 \\
        \midrule
        -- & Spectral reference
        & 1.172 & 5.528 & 1.340 \\
        \bottomrule
    \end{tabular}
\end{table}

\paragraph{Optimization trajectories.}
At $m=128$, all Lanczos variants remain accurate over the eight Adam
updates (Table~\ref{tab:regimes-fit}). At $m=1536$, however,
unreorthogonalized AD reaches a maximum gradient error of
$2.72\times10^5$ while maintaining a maximum value error of
$4.04\times10^{-7}$. Its scale trajectory departs from the spectral
reference (Figure~\ref{fig:regimes-fit}), ending at approximately
$s=4.64$ compared with $s=5.53$ for the reference.

The projected method follows the reference trajectory, with a maximum
gradient error of $1.82\times10^{-7}$. Its mean data-term
value-and-gradient evaluation time is $5.201$ s, compared with
$9.146$ s for unreorthogonalized AD and $9.892$ s for the Chebyshev
configuration used in this run. Chebyshev also follows the reference
trajectory, with a maximum gradient error of $7.86\times10^{-5}$.
Table~\ref{tab:regimes-fit} reports the final values of all three
parameters. These short trajectories demonstrate the effect of gradient
errors on the optimizer; agreement of the stored parameter coordinates
does not imply accuracy beyond the float32 operator evaluations.

\begin{figure}[t]
    \centering
    \includegraphics[width=\linewidth]{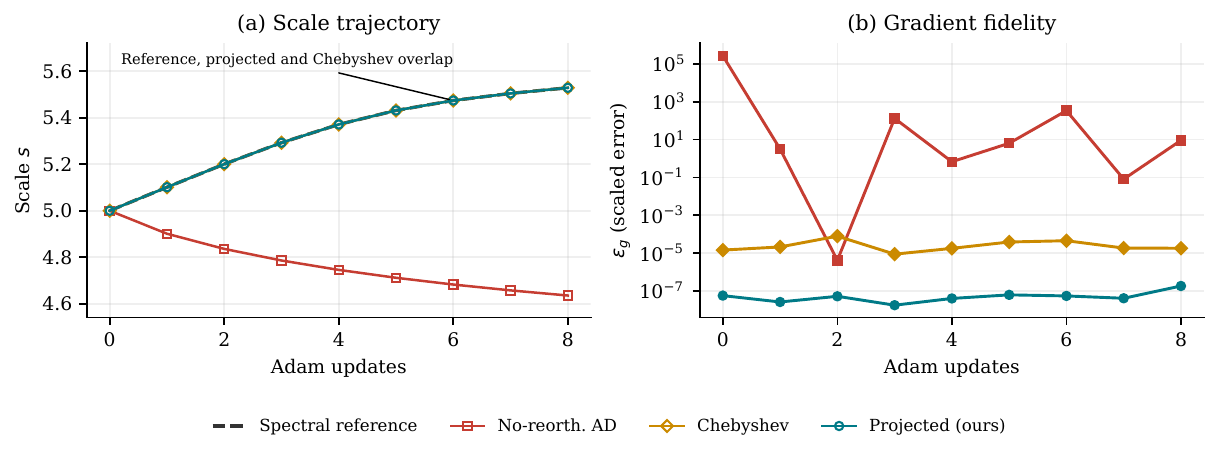}
    \caption{
        Optimization at Lanczos depth $m=1536$.
        Left: scale trajectories over eight Adam updates; the
        projected, Chebyshev, and spectral-reference curves nearly
        coincide. Right: objective-gradient errors evaluated at
        each method's current parameters. The Chebyshev degree is
        unspecified in the available record. Full-reorthogonalized
        AD is omitted.
    }
    \label{fig:regimes-fit}
\end{figure}

\paragraph{Componentwise gradient errors.}
To separate the error in the data gradient from normalization by the
full objective gradient, define
$\Delta g=\widetilde{\nabla_\theta\mathcal L}
-\nabla_\theta\mathcal L$ and
\begin{equation}
    \eta_D=\frac{\|\Delta g\|_2}{\|\nabla_\theta D\|_2}.
    \label{eq:regimes-data-gradient-error}
\end{equation}
The denominator is nonzero at every evaluated iterate. At $m=1536$,
the projected method has maximum data-gradient-relative error
$1.73\times10^{-6}$ (Table~\ref{tab:regimes-components}), compared
with the objective-scaled error $1.82\times10^{-7}$.
For unreorthogonalized AD, the absolute error in the $\log s$
component reaches $7.22\times10^{10}$. Thus the degradation is also
evident in the individual parameter derivatives and under normalization
by the data-gradient norm.

\begin{table}[t]
    \centering
    \small
    \caption{
        Componentwise absolute gradient errors and data-gradient-relative
        errors at $m=1536$. Derivatives are taken with respect to
        $(\log\mu,\log s,\alpha)$. Each column reports its maximum
        over the nine evaluations, which may occur at different iterates.
    }
    \label{tab:regimes-components}
    \setlength{\tabcolsep}{3.5pt}
    \begin{tabular}{lrrrr}
        \toprule
        Method
        & $|\Delta g_{\log\mu}|$
        & $|\Delta g_{\log s}|$
        & $|\Delta g_\alpha|$
        & $\eta_D$ \\
        \midrule
        No-reorth. AD
        & $5.65\times10^{5}$
        & $7.22\times10^{10}$
        & $1.88\times10^{-2}$
        & $2.44\times10^{6}$ \\
        Projected (ours)
        & $2.04\times10^{-3}$
        & $1.37\times10^{-2}$
        & $4.67\times10^{-2}$
        & $1.73\times10^{-6}$ \\
        Chebyshev
        & $1.54\times10^{-2}$
        & $1.38\times10^{-2}$
        & $2.08\times10^{1}$
        & $6.96\times10^{-4}$ \\
        \bottomrule
    \end{tabular}
\end{table}

\paragraph{Storage requirements.}
The affine specialization uses $O(n+m^2)$ working storage for a fixed
number of samples, since the gradient contractions require only the
tridiagonal matrices and their sensitivities. The general stored-basis
implementation instead uses $O(nm+m^2)$, excluding operator and VJP
workspaces.

For comparison, an AD implementation that retains separately
materialized float32 basis blocks of sizes $n\times j$ for
$j=1,\ldots,m$ requires
\begin{equation}
    4n\sum_{j=1}^m j=2nm(m+1)\ \text{bytes}
    \label{eq:regimes-memory-model}
\end{equation}
for these blocks alone. At $n=32{,}768$ and $m=768$, this amounts to
$36.047$ GiB, compared with $96$ MiB for a single basis
(Table~\ref{tab:regimes-memory}). These counts describe a particular
allocation strategy; checkpointing, buffer reuse, or custom adjoints
can reduce the storage required by differentiated Lanczos.
A full-reorthogonalized AD evaluation at $m=768$ was reported to
exhaust memory, but its failure log and launch configuration are
unavailable. Peak memory was not measured in the completed runs, so
the analytical counts do not establish the peak allocation or cause
of that failure.

\begin{table}[t]
    \centering
    \caption{
        Analytical storage requirements for $n=32{,}768$. The basis and
        growing stacks use float32; the small $m\times m$ tensor uses
        float64. Entries give tensor payloads, excluding other arrays
        and workspaces, and are not measured peak-memory values.
    }
    \label{tab:regimes-memory}
    \small
    \setlength{\tabcolsep}{5pt}
    \begin{tabular}{rrrr}
        \toprule
        $m$ & One basis (MiB)
        & Growing stacks (GiB) & One small tensor (MiB) \\
        \midrule
        128  & 16  & 1.008   & 0.125 \\
        384  & 48  & 9.023   & 1.125 \\
        768  & 96  & 36.047  & 4.500 \\
        1536 & 192 & 144.094 & 18.000 \\
        \bottomrule
    \end{tabular}
\end{table}

\end{document}